\documentclass[11pt,reqno,
]{amsart}
\usepackage{xcolor}
\definecolor{winered}{rgb}{0.6,0,0}
\definecolor{lessblue}{rgb}{0,0,0.7}
\usepackage[pdftex,colorlinks=true,linkcolor=winered,citecolor=lessblue,breaklinks=true,bookmarksopen=true]{hyperref}
\usepackage{
amssymb, graphicx
}
\usepackage{mathrsfs} 
\usepackage{amsthm}
\usepackage{graphicx,color}

\numberwithin{equation}{section}
\numberwithin{figure}{section}

\newtheorem{theorem}{Theorem}[section]
\newtheorem{lemma}[theorem]{Lemma}
\newtheorem{prop}[theorem]{Proposition}

\theoremstyle{definition}

\theoremstyle{remark}
\newtheorem{remark}[theorem]{Remark}

\date{\today}

\newcommand{\pa}{\partial}
\newcommand{\wt}{\widetilde}
\newcommand{\wh}{\widehat}
\newcommand{\eps}{\epsilon}
\newcommand{\dd}{\mathrm{d}}
\newcommand{\CI}{{\mathcal C^\infty}}
\newcommand{\CIc}{{\mathcal C^\infty_{\mathrm c}}}
\newcommand{\cC}{{\mathcal C}}
\newcommand{\cE}{{\mathcal E}}
\newcommand{\cI}{{\mathcal I}}
\newcommand{\cK}{{\mathcal K}}
\newcommand{\cQ}{{\mathcal Q}}
\newcommand{\cR}{{\mathcal R}}
\newcommand{\cS}{{\mathcal S}}
\newcommand{\cU}{{\mathcal U}}
\newcommand{\cV}{{\mathcal V}}
\newcommand{\cW}{{\mathcal W}}
\newcommand{\ft}{{\mathfrak t}}
\newcommand{\bft}{{\mathbf t}}
\newcommand{\R}{{\mathbb R}}
\DeclareMathOperator{\supp}{supp}

\title[DN map for a semilinear wave equation]{The Dirichlet-to-Neumann map for a semilinear wave equation on Lorentzian manifolds}

\author{Peter Hintz}
  \address{Department of Mathematics, Massachusetts Institute of Technology, Cambridge, MA 02139,USA
  (\tt{phintz@mit.edu})
  }
  \author{Gunther Uhlmann}
\address{Department of Mathematics, University of Washington, Seattle, WA 98195, USA; Institute for Advanced Study, 
The Hong Kong University of Science and Technology, Kowloon, Hong Kong, China (\tt{gunther@math.washington.edu})
}
  
  \author{Jian Zhai}
\address{Institute for Advanced Study,
  The Hong Kong University of Science and Technology, Kowloon, Hong Kong, China
  (\tt{iasjzhai@ust.hk})}
 
\begin{document}

\begin{abstract}
  We consider the semilinear wave equation $\Box_g u+a u^4=0$, $a\neq 0$, on a Lorentzian manifold $(M,g)$ with timelike boundary. We show that from the knowledge of the Dirichlet-to-Neumann map one can recover the metric $g$ and the coefficient $a$ up to natural obstructions. Our proof rests on the analysis of the interaction of distorted plane waves together with a scattering control argument, as well as Gaussian beam solutions.
\end{abstract}


\maketitle

\section{Introduction}

Let $(M,g)$ be a $(1+3)$-dimensional Lorentzian manifold with boundary $\partial M$, where the metric $g$ is of signature $(-,+,+,+)$. We assume that $M=\mathbb{R}\times N$ where $N$ is a connected manifold with boundary $\partial N$, and write the metric $g$ as
\begin{equation}
\label{EqIMetric}
  g =-\alpha(\ft,x')\mathrm{d}\ft^2+\kappa(\ft,x'),
\end{equation}
where $x=(\ft,x')=(x^0,x^1,x^2,x^3)$ are local coordinates on $M$; here, $\alpha:\mathbb{R}\times N\rightarrow(0,\infty)$ is a smooth function and $\kappa(\ft,\,\cdot\,)$ is a Riemannian metric on $N$ depending smoothly on $\ft\in \mathbb{R}$. The boundary $\partial M=\mathbb{R}\times\partial N$ of $M$ is then timelike. Let $\nu=\nu_g$ denote the unit outer normal vector field to $\partial M$ with respect to the metric $g$. Assume that $\partial M$ is null-convex, which means that $\mathrm{II}(V,V)=g(\nabla_V\nu,V)\geq 0$ for all null vectors $V\in T(\partial M)$; see \cite{hintz2017reconstruction} for a discussion of this condition. We consider the semilinear wave equation
\begin{alignat}{2}
  \Box_g u(x)+a(x)u(x)^4&=0,&\quad&\text{on }M,\nonumber\\
\label{maineq}
  u(x)&=f(x),&\quad& \text{on }\partial M,\\
  u(\ft,x')&=0, &\quad& \ft<0,\nonumber
\end{alignat}
where $\Box_g=(-\det g)^{-1/2}\partial_j((-\det g)^{1/2}g^{j k}\partial_k)$ is the wave operator (d'Alembertian) on $(M,g)$, and $a\in\CI(M)$ is a nowhere vanishing function (thus either everywhere positive or everywhere negative).

We only consider boundary sources $f$ with $\mathrm{supp}\,f\subset (0,T)\times\partial N$ and introduce the Dirichlet-to-Neumann (DN) map $\Lambda_{g,a}$ (measured in $(0,T)\times\partial N$) defined as
\[
  \Lambda_{g,a} f = \partial_\nu u\vert_{(0,T)\times\partial N}=\nu^j \partial_j u\vert_{(0,T)\times\partial N},
\]
where $u$ is the solution of \eqref{maineq}. The well-posedness of the initial boundary value problem \eqref{maineq} with small Dirichlet data $f\in \mathcal{C}^m$, $m\geq 6$, can be established as in \cite{hintz2020inverse}. Thus $\Lambda_{g,a}f$ is well-defined for such $f$. We will study the inverse problem or recovering the Lorentzian metric $g$ and the nonlinear coefficient $a$ from $\Lambda_{g,a}$. We choose to consider the semilinear equation with a quartic nonlinear term because it requires the least amount of technicalities in the analysis of nonlinear interactions of distorted plane waves or Gaussian beams. Much as in the setting of manifolds without boundary (see the references below), more general nonlinearities can be dealt with in a similar way.

Since the work \cite{kurylev2018inverse}, rapid progress has been made on the study of inverse problems for nonlinear equations. See \cite{lassas2018inverse,kurylev2014inverse,lassas2017determination,de2019nonlinear,wang2019inverse,de2018nonlinear,uhlmann2018determination,chen2019detection,feizmohammadi2019recovery,hintz2020inverse,chen2020inverse,balehowsky2020inverse,lassas2020uniqueness,feizmohammadi2020inverse,lai2021reconstruction} for results on hyperbolic equations and \cite{assylbekov2017direct,carstea2019reconstruction,carstea2020inverse,lassas2019inverse,lassas2019partial,feizmohammadi2019inverse,krupchyk2020remark,krupchyk2019partial,assylbekov2020inverse,krupchyk2020inverse,kian2020partial,carstea2021inverse} for elliptic equations. For hyperbolic equations, the recovery of time-dependent coefficients is possible for some nonlinear equations, whereas the corresponding problems for linear equations are still largely open. See \cite{uhlmann2020inverse} for an overview of the recent progress on inverse problems for nonlinear hyperbolic equations.

The recovery of a Lorentzian metric from the \textit{source-to-solution map} associated with a semilinear equation is considered in \cite{kurylev2018inverse}, where the authors use the nonlinear interactions of distorted plane waves to recover the light observation sets, and use them to reconstruct the Lorentzian geometry. The approach was then further generalized to deal with other different types of nonlinear equations in \cite{lassas2018inverse,kurylev2014inverse,lassas2017determination,wang2019inverse,uhlmann2018determination,balehowsky2020inverse}; recently \cite{tzou2021point} it was shown that recovery of a Riemannian metric is possible when one measures the solution of a forced semilinear wave equation only at a single point for some time. Distorted plane waves can also be used to recover the coefficients of (linear) lower order terms \cite{chen2019detection,chen2020inverse} and the nonlinear terms \cite{lassas2018inverse,de2018nonlinear,hintz2020inverse}. In the works \cite{feizmohammadi2019recovery,uhlmann2019inverse,hintz2020inverse,feizmohammadi2020inverse}, Gaussian beams, instead of distorted planes waves, are used to study inverse problems for nonlinear wave equations. We will use both distorted plane waves and Gaussian beams in this paper.

First, notice that there a natural geometric obstruction to the reconstruction of $g$ and $a$:
\begin{lemma}
If $\Psi:M\rightarrow M$ is a diffeomorphism and $\Psi\vert_{\partial M}=\mathrm{Id}$, then
\[
  \Lambda_{g,a}=\Lambda_{\Psi^*g,\Psi^*a}.
\]
\end{lemma}
\begin{proof}
  If $\Box_g u+a u^4=0$, then
  \[
    0=\Psi^*(\Box_g u+a u^4)=\Box_{\Psi^*g}(\Psi^*u)+(\Psi^*a)(\Psi^*u)^4.
  \]
  Now, if $u|_{\pa M}=f$, then $(\Psi^*u)|_{\pa M}=f$ also; moreover, the outer unit normals at a point $x\in\pa M$ with respect to $g$ and $\Psi^*g$ are equal to $\nu$ and $\Psi^*\nu=(\Psi^{-1})_*\nu$, respectively. Therefore, writing $\pa_{\Psi^*\nu}$ for differentiation along $\Psi^*\nu$, we compute at $x\in\pa M$
  \[
    (\Lambda_{\Psi^*g,\Psi^*a}f)(x) = \bigl(\pa_{\Psi^*\nu}(\Psi^*u)\bigr)(x) = \bigl(\Psi^*(\pa_\nu u)\bigr)(x) = \pa_\nu u(x) = (\Lambda_{g,a}f)(x);
  \]
  the penultimate equality uses that $\Psi(x)=x$.
\end{proof}

There is another invariance of the DN map $\Lambda_{g,a}$ related to a conformal change of the metric $g$. 
\begin{lemma}\label{conf_invar}
Let $\beta$ be a smooth function on $M$, such that
\begin{equation}\label{eq_beta}
\beta\vert_{\partial M}=0,\quad\partial_\nu\beta\vert_{\partial M}=0,\quad \square_g e^{-\beta}=0,
\end{equation}
then we have
\[
\Lambda_{e^{-2\beta}g,e^{-\beta}a}=\Lambda_{g,a}.
\]
\end{lemma}
\begin{proof}
The key formula is the well-known
\begin{equation}
\label{EqConf}
  e^{3\beta}\Box_g(e^{-\beta}v) = \Box_{e^{-2\beta}g}v + \tilde q v,\quad \tilde q=e^{3\beta}\Box_g(e^{-\beta}),
\end{equation}
valid for any Lorentzian metric $g$ on a 4-dimensional spacetime. Indeed, the operators $e^{3\beta}\Box_g e^{-\beta}$ and $\Box_{e^{-2\beta}g}+\tilde q$ are both symmetric with respect to the volume density $|\mathrm{d}(e^{-2\beta}g)|=e^{-4\beta}|\mathrm{d}g|$ and have the same principal symbol, hence their difference, being a symmetric first order operator with real coefficients, is in fact an operator of order $0$, and since both operators give the same result when acting on the constant function $1$, the formula~\eqref{EqConf} follows.

Denote now $\widetilde{g}=e^{-2\beta}g$, $\widetilde{a}=e^{-\beta}a$ with $\beta$ satisfying \eqref{eq_beta}. Let $v=e^{\beta}u$, then we have
\[
\square_{\widetilde{g}}v+\widetilde{a}v^4=\square_{e^{-2\beta}g}(e^{\beta}u)+\widetilde{a}e^{4\beta}u^4=e^{3\beta}(\square_gu+qu+a u^4)=e^{3\beta}(\square_gu+a u^4),
\]
where $q=-e^\beta\Box_g e^{-\beta}=0$ in $M$. Notice that, by the assumption on $\beta$, $\nu_{e^{-2\beta}g}=\nu_g$, and if $u\vert_{\partial M}=f$, then $v\vert_{\partial M}=f$ also. Moreover, we have
\[
\Lambda_{e^{-2\beta}g,e^{-\beta}a}f=\partial_\nu v=\partial_\nu u=\Lambda_{g,a}f.\qedhere
\]
\end{proof}

Recall that a smooth curve $\mu:(a,b)\rightarrow M$ is called causal if $g(\dot{\mu}(s),\dot{\mu}(s))\leq 0$ and $\dot{\mu}(s)\neq 0$ for all $s\in(a,b)$. Given $p,q\in M$, we denote $p\leq q$ if $p= q$ or $p$ can be joined to $q$ by a future-pointing causal curve. We say $p<q$ if $p\leq q$ and $p\neq q$. We denote the causal future of $p\in M$ by $J_g^+(p)=\{q\in M: p\leq q\}$ and the causal past of $q\in M$ by $J_g^-(q)=\{p\in M: p\leq q\}$. The curve $\mu$ is called time-like if $g(\dot{\mu}(s),\dot{\mu}(s))< 0$ for all $s\in(a,b)$. We denote $p\ll q$ if $p\neq q$ and there is a future-pointing time-like path from $p$ to $q$. Then the chronological future of $p\in M$ is the set $I_g^+(p)=\{q\in M \colon p\ll q\}$ and the chronological past of $q\in M$ is $I_g^-(q)=\{p\in M \colon p\ll q\}$. We also denote $J_g(p,q):=J^+_g(p)\cap J^-_g(q)$ and $I_g(p,q):=I^+_g(p)\cap I^-_g(q)$.
We will reconstruct the metric $g$ and the coefficient $a$ in the set
\begin{equation}
\label{EqDiamond}
  \mathbb{U}_g=\bigcup_{p,q\in (0,T)\times\partial N}J_g(p,q).
\end{equation}
We shall make the simplifying assumption that no null-geodesics in $\mathbb{U}_g$ has conjugate points. The main result of this paper is as follows.

\begin{theorem}
\label{ThmI}
 Suppose we are given two manifolds $M,\widetilde{M}$ smooth Lorentzian metrics $g,\widetilde{g}$ on $M,\widetilde{M}$ of the form~\eqref{EqIMetric} with respect to product decompositions $M=\R\times N$, $\widetilde{M}=\R\times\widetilde{N}$. Suppose $\partial N\cong\partial\widetilde{N}$, and identify $\R\times\partial N$ with $\R\times\partial\widetilde{N}$. Assume that null-geodesics in $(\mathbb U_g,g)$, resp.\ $(\mathbb U_{\wt g},\wt g)$ have no conjugate points. Suppose moreover we are given smooth nonvanishing functions $a,\widetilde{a}$ on $M,\widetilde{M}$. Assume that the Dirichlet-to-Neumann maps agree, that is, $\Lambda_{g,a}=\Lambda_{\widetilde{g},\widetilde{a}}$. Then there exist a diffeomorphism $\Psi\colon\mathbb U_g\to\mathbb U_{\wt g}$ with $\Psi\vert_{(0,T)\times\partial N}=\mathrm{Id}$ and a smooth function $\beta\in\CI(M)$, $\beta\vert_{(0,T)\times\partial N}=\partial_\nu\beta\vert_{(0,T)\times\partial N}=0$, so that, in $\mathbb U_g$,
  \[
    \Psi^*\widetilde{g}=e^{-2\beta}g,\quad
    \Psi^*\widetilde{a}=e^{-\beta}a,\quad \square_g e^{-\beta}=0.
  \]
  In particular, if $a=\wt a=c\neq 0$ are equal to the same nonzero constant, then $\beta=0$ and $\Psi^*\wt g=g$ in $\mathbb{U}_g$, that is, $\Psi$ is an isometry.
\end{theorem}

\begin{remark}
We can not conclude from the result of Theorem \ref{ThmI} that $\beta=0$ in $\mathbb{U}_g$, if the \textit{Unique Continuation Principle} (UCP) fails for the operator $\square_g$. If there exists a function $\beta\in\CI(M)$ as in the statement of the Theorem, then $\Lambda_{e^{-2\beta}g,e^{-\beta}a}=\Lambda_{g,a}$ by Lemma \ref{conf_invar},  giving a counter-example for the unique determination of $g$ and $a$ (up to diffeomorphism) in $\mathbb{U}_g$ from $\Lambda_{g,a}$. In this sense, our result is sharp. We refer to \cite{alinhac1995non} for the failure of UCP for an operator of the form $\square_g+b$. It is an open problem to prove UCP or to construct a counterexample to UCP for $\Box_g$.
\end{remark}

The structure of this paper is as follows. In Section \ref{boundary_determination}, we show the determination of the jet of the metric $g$ at the boundary $\partial M$. In Section \ref{interiordetermination}, we present the proof of the determination of the conformal class of the metric $g$ in the interior of the manifold. The nonlinear interaction of distorted plane waves is used for the proof. Due to the complications caused by possible reflections of waves at the boundary, we propose a ``scattering control" scheme to suppress the reflections. In Section \ref{gaussianbeam}, we use Gaussian beam solutions to prove the uniqueness of the coefficient $a$ and the conformal factor up to natural obstructions.

\section{Boundary determination}\label{boundary_determination}

In this section, we consider the determination of the jet of the metric on the boundary $\pa M$. This only uses the Dirichlet-to-Neumann map $\Lambda^{\rm lin}_g$ for the linearized equation
\begin{alignat}{2}
\Box_g u(x)&=0,&\quad&\text{on }M,\nonumber\\
\label{linear_eq}
 u(x)&=f(x),&\quad& \text{on }\partial M,\\
u(\ft,x')&=0, &\quad& \ft<0.\nonumber
\end{alignat}
Note that we can recover $\Lambda^{\rm lin}_g$ from $\Lambda_{g,a}$ (without any restrictions on $a$ except for smoothness) via
\[
\Lambda_g^{\rm lin}f=\frac{\pa}{\pa\epsilon}\bigl(\Lambda_{g,a}(\epsilon f)\bigr)\Big|_{\epsilon=0}.
\]
In particular, the assumptions of Theorem~\ref{ThmI} imply that $\Lambda_g^{\rm lin}=\Lambda_{\wt g}^{\rm lin}$.

\emph{The discussion in the present section works in $(1+n)$ dimensions for any $n\geq 1$, and we shall proceed in this generality.} We will recover $g$ in Taylor series at $(0,T)\times\pa N$ from $\Lambda_g^{\rm lin}$ in boundary normal coordinates. Here, we recall:

\begin{lemma}{$($\cite{petrov2016einstein}$)$}
\label{LemmaBdyNormal}
  Let $U\subset\pa M$ be an open precompact subset. Then there exist $\eps>0$, a neighborhood $V\subset M$ of $U$ in $M$, and a diffeomorphism $\Psi\colon(\pa M\cap V)\times[0,\eps)\to V$ such that
  \begin{enumerate}
  \item $\Psi(x',0)=x'$ for all $x'\in\pa M\cap V$;
  \item $\Psi(x',x^n)=\gamma_{x'}(x^n)$, where $\gamma_{x'}(x^n)$ is the unit speed geodesic (with respect to $g$) issued from $x'$ normal to $\pa M$.
  \end{enumerate}
  The pullback metric $\Psi^*g$ takes the form
  \begin{equation}
  \label{EqBdyNorm}
    \Psi^*g(x',x^n) = k(x^n,x',\dd x') + \dd x^n\otimes\dd x^n,
  \end{equation}
  where $k(x^n,x',\dd x')$ is a Lorentzian metric on $\pa M\cap V$ depending smoothly on $x^n\in[0,\eps)$. In particular, if $(x^0,\ldots, x^{n-1})$ are local boundary coordinates on $\partial M$, then in the coordinate system $(x^0,\ldots, x^n)$, the metric tensor $g$ takes the form
  \begin{equation}
  \label{EqSemigeod}
  g=\sum_{0\leq\alpha,\beta\leq n-1}g_{\alpha\beta}\mathrm{d}x^\alpha\otimes\mathrm{d}x^\beta+\mathrm{d}x^n\otimes\mathrm{d}x^n.
  \end{equation}
\end{lemma}

Indeed, we have $U\subset[a,b]\times\pa N$ for some $-\infty<a<b<\infty$; the existence of the unit speed geodesics $\gamma_{x'}(x^n)$ for $x^n\in[0,\eps')$ for some small $\eps'>0$ (depending on $a,b$) then follows from basic ODE theory, as does the fact that map $\Psi(x',x^n)=\gamma_{x'}(x^n)$ is a diffeomorphism for $x'\in U$ and $x^n\in[0,\eps)$ for sufficiently small $\eps\in(0,\eps']$. See also~\cite[Lemma 2.3]{stefanov2018inverse}.

We are interested in the special case $U=(0,T)\times\pa N$. We refer to the partial coordinate system $(x',x^n)$ near $(0,T)\times\pa N$, or the full local coordinate system $(x^0,\ldots,x^n)$, as \emph{boundary normal coordinates}. Assume we have two metrics $g,\widetilde{g}$ such that $\Lambda_g^{\rm lin}=\Lambda_{\widetilde{g}}^{\rm lin}$. Let $\Psi$ and $\widetilde{\Psi}$ be the diffeomorphisms as in Lemma \ref{LemmaBdyNormal} near $(-1,T+1)\times\partial N$ with respect to $g$ and $\widetilde{g}$. (Recall that we are identifying $\partial{N}$ and $\partial\widetilde{N}$.) Then $\widetilde{\Psi}\circ\Psi^{-1}$ is a diffeomorphism near $(0,T)\times\partial N$ fixing the boundary. Thus, $g$ and $(\widetilde{\Psi}\circ\Psi^{-1})^*\widetilde{g}$ have a common boundary normal coordinate system near $\partial M=\partial{\widetilde M}$. Without loss of generality, we shall thus assume that $g$ and $\widetilde{g}$ have the same boundary normal coordinates, i.e.\ they take the form~\eqref{EqSemigeod} (with a priori possibly different coefficients $g_{\alpha\beta}$, $\wt g_{\alpha\beta}$) in the \emph{same} (partial) coordinate system near $x_0\in\pa M=\pa\widetilde{M}$.

\begin{theorem}\label{bdrydetermination}
  If $\Lambda_g^{\rm lin}=\Lambda_{\widetilde{g}}^{\rm lin}$, then $g$ and $\widetilde{g}$ in boundary normal coordinates are equal in Taylor series at $(0,T)\times\partial N$.
\end{theorem}
\begin{proof}
 The proof is adapted from the proof of Theorem 3.2 in \cite{stefanov2018inverse}. We work in a fixed boundary normal coordinate system $x=(x',x^n)$ in a neighborhood of a point $x_0\in(0,T)\times\partial N\subset\partial M$; denote the corresponding coordinates on $T^*M$ by $\xi=(\xi',\xi_n)$. Let $\xi^{0\prime}\in\R^{n}$ be such that $(\xi^{0\prime},0)\in T_{x_0}^*M$ is a future-pointing timelike covector (which in our special coordinates is thus dual to a tangent vector to $\pa M$), and denote by $\mathcal{U}\subset T^*(\partial M)$ a small neighborhood of $(x_0,\xi^{0\prime})$ so that for all  $(x',\xi')\in\mathcal{U}$, the covector $\xi'$ is future timelike. Let $\chi(x',\xi')$ be a smooth cutoff function with support in $\mathcal{U}$ and $\chi=1$ near $(x_0,\xi^{0\prime})$.
  
  For any fixed $\xi'\in\R^{n}$, consider then the Dirichlet data
  \[
  f(x')=e^{\mathrm{i}\lambda x'\cdot \xi'}\chi(x',\xi'),
  \]
  with $\lambda$ a large parameter. Then we construct (outgoing) geometric optics solutions near $x_0$ to the equation \eqref{linear_eq} of the form
  \[
  u(x)\sim e^{\mathrm{i}\lambda\phi(x,\xi')}\sum_{j=0}^\infty\lambda^{-j}a_j(x,\xi'),
  \]
  where the notation `$\sim$' means that the difference of $u$ and the truncated sum over $j=0,\ldots,J$ is bounded in $L^\infty$ by $\lambda^{-J-1}$.
  
  By the discussion in \cite{stefanov2018inverse}, we have the eikonal equation for the phase function
  \begin{equation}\label{eq_phase}
  g^{\alpha\beta}\partial_\alpha\phi\partial_\beta\phi+(\partial_n\phi)^2=0,\quad \phi\vert_{x^n=0}=x'\cdot\xi',
  \end{equation}
  and the transport equations for the amplitudes
  \begin{subequations}
  \begin{alignat}{2}
  Ta_0&=0,&\quad a_0\vert_{x^n=0}&=\chi\label{eq_a0}\\
  \mathrm{i}Ta_j&=-\Box_g a_{j-1},&\quad a_j\vert_{x^n=0}&=0,\quad j\geq 1,\label{eq_aj}
  \end{alignat}
  \end{subequations}
  where $T$ is the transport operator defined by
  \begin{equation}\label{transport_T}
  T a:=2g^{jk}\partial_j\phi\partial_k a+(\Box_g\phi)a = 2\pa_n\phi\pa_n a + 2 g^{\alpha\beta}\pa_\alpha\phi\pa_\beta a + (\Box_g\phi)a.
  \end{equation}
  Note that $\pa_n$ is the \emph{inward} pointing unit normal, so $\pa_n=-\pa_\nu$. Now, with the outgoing condition $\partial_n \phi\vert_{\partial M}<0$, the equation \eqref{eq_phase} implies that
  \begin{equation}\label{formula_xi_n}
  \partial_n\phi((x',0),\xi')=\xi_n(x',\xi'):=-\sqrt{-g^{\alpha\beta}(x',0)\xi_\alpha\xi_\beta}<0.
  \end{equation}
  
  In our boundary normal coordinates, we have
  \begin{equation}
  (\Lambda^{\rm lin}_gf)(x')\sim-e^{\mathrm{i}\lambda x'\cdot\xi'}\Big(\mathrm{i}\lambda\partial_n\phi(x',0,\xi')+\sum_{j=0}^\infty\lambda^{-j}\partial_na_j(x',0,\xi')\Big).
  \end{equation}
  The expression for$ \Lambda^{\rm lin}_{\widetilde{g}}f$ is similar, with $\phi$ and $a_j$ replaced by $\widetilde{\phi}$ and $\widetilde{a}_j$.
  Thus, under the assumption that $\Lambda_g^{\rm lin}f=\Lambda_{\widetilde{g}}^{\rm lin}f$, we have
  \begin{subequations}
  \begin{align}
  \partial_n\phi((x',0),\xi')&=\partial_n\widetilde{\phi}((x',0),\xi')\label{identity_phase}\\
  \partial_na_j((x',0),\xi')&=\partial_n\widetilde{a}_j((x',0),\xi'),\quad j=0,1,\ldots.\label{identity_amplitude}
  \end{align}
  \end{subequations}
  By \eqref{identity_phase} and \eqref{eq_phase}, we have (with $\xi_\alpha=\xi'_\alpha$ for $\alpha=0,\ldots,n-1$)
  \[
  g^{\alpha\beta}\xi_\alpha\xi_\beta=\widetilde{g}^{\alpha\beta}\xi_\alpha\xi_\beta,
  \]
  for $\xi'$ in an open subset of $\R^{n}$. This implies $g\vert_{\partial M}=\widetilde{g}\vert_{\partial M}$.
  
  Next, the transport equation \eqref{eq_a0} for $a_0$ gives the identity
  \begin{equation}\label{a_0_eq_ex}
  \begin{split}
    0 &= 2\partial_n\phi\partial_na_0+\frac{1}{\sqrt{-\det g}}\partial_n\bigl(\sqrt{-\det g}\partial_n\phi\bigr)+Q(g) \\
      &= 2\pa_n\phi\pa_n a_0 + \frac{1}{2\det g}(\partial_n\det g)\partial_n\phi+\partial^2_n\phi + Q(g)
  \end{split}
  \end{equation}
  on $\{x' \colon \chi(x',\xi')=1\}$ (which contains an open neighborhood of $x_0$), since $\partial_\beta a_0=\partial_\beta\chi=0$ there; here, $Q$ is given by
  \[
  Q(g)=\frac{1}{\sqrt{-\det g}}\partial_\alpha\bigl(\sqrt{-\det g}g^{\alpha\beta}\bigr)\xi_\beta
  \]
  and thus depends only on $g\vert_{\partial M}$, which is already determined. Differentiating \eqref{eq_phase} along $\pa_n$, we obtain
  \begin{equation}\label{pn_phi_identity}
  2\xi_n\partial^2_n\phi=-\partial_ng^{\alpha\beta}\xi_\alpha\xi_\beta+R(g,\phi,\partial_n\phi), \quad \text{on }\partial M,
  \end{equation}
  where $R(g,\phi,\pa_n\phi)$ depends on $g\vert_{\partial},\phi\vert_{\partial M},\pa_n\phi|_{\pa M}$ only. Using this to eliminate $\pa_n^2\phi$ in the identity \eqref{identity_amplitude} for $j=0$, multiplying by $2\xi_n=2\pa_n\phi((x',0),\xi')=2\pa_n\wt\phi((x',0),\xi')$, and using that $\pa_n\phi\pa_n a_0=\pa_n\wt\phi\pa_n\wt a_0$, we thus obtain
  \[
  \frac{\xi_n^2}{\det g}\partial_n\det g-\partial_ng^{\alpha\beta}\xi_\alpha\xi_\beta+\mathcal{R}(g,\phi,\partial_n\phi)=\frac{\xi_n^2}{\det \widetilde{g}}\partial_n\det \widetilde{g}-\partial_n\widetilde{g}^{\alpha\beta}\xi_\alpha\xi_\beta+\mathcal{R}(\widetilde{g},\widetilde{\phi},\partial_n\widetilde{\phi}),
  \]
  where $\mathcal{R}(g,\phi,\partial_n\phi)$ depends on $g\vert_{\partial M}$, $\phi\vert_{\partial M}$, $\partial_n\phi\vert_{\partial M}$ only. Since $g\vert_{\partial M}=\widetilde{g}\vert_{\partial M}$, $\phi\vert_{\partial M}=\widetilde{\phi}\vert_{\partial M}$, $\partial_n\phi\vert_{\partial M}=\partial_n\widetilde{\phi}\vert_{\partial M}$, we have $\cR(g,\phi,\pa_n\phi)=\cR(\wt g,\wt\phi,\pa_n\wt\phi)$. Using \eqref{formula_xi_n}, we thus get
  \[
  -\frac{1}{\det g}(\partial_n\det g)g^{\alpha\beta}\xi_\alpha\xi_\beta-\partial_ng^{\alpha\beta}\xi_\alpha\xi_\beta=-\frac{1}{\det \widetilde{g}}(\partial_n\det \widetilde{g})\widetilde{g}^{\alpha\beta}\xi_\alpha\xi_\beta-\partial_n\widetilde{g}^{\alpha\beta}\xi_\alpha\xi_\beta,
  \]
  or equivalently
  \begin{equation}
  \frac{1}{\det g}\partial_n((\det g) g^{\alpha\beta})=\frac{1}{\det \widetilde{g}}\partial_n((\det \widetilde{g}) \widetilde{g}^{\alpha\beta})\quad \text{on }\partial M.
  \end{equation}
  Now, $\det g=\det \widetilde{g}$ on $\partial M$, thus for $h^{\alpha\beta}:=(\det g)g^{\alpha\beta}=\wt h^{\alpha\beta}$ we have $\partial_nh^{\alpha\beta}=\partial_n\widetilde{h}^{\alpha\beta}$ on $\partial M$. Since $h_{\alpha\beta}=(\det g)^{-1}g_{\alpha\beta}$ with $\det h=(\det g)^{1-n}$, we thus have
  \[
    \partial_n g^{\alpha\beta}=\partial_n((\det h)^{\frac{1}{1-n}}h^{\alpha\beta})=\partial_n((\det\widetilde{h})^{\frac{1}{1-n}}\widetilde{h}^{\alpha\beta})=\partial_n \widetilde{g}^{\alpha\beta}.
  \]
  From the identity~\eqref{pn_phi_identity}, we can then also conclude that $\partial_n^2\phi=\partial_n^2\widetilde{\phi}$ on $\partial M$.
  
  Next, let us consider the transport equation \eqref{eq_aj} for $j=1$. Since $a_1=0$ on $\partial M$, we have
  \[
  2\mathrm{i}\partial_na_1=-\Box_ga_0=-\frac{1}{\sqrt{-\det g}}\partial_n\bigl(\sqrt{-\det g}\partial_n a_0\bigr)=-\partial^2_na_0+Q(g,\partial_ng).
  \]
  By differentiating \eqref{a_0_eq_ex} once and \eqref{eq_phase} twice, we have
  \[
  \begin{split}
  2\mathrm{i}\partial_na_1&=\frac{1}{4\det g}(\partial^2_n\det g)+\frac{1}{2\partial_n\phi}\partial^3_n\phi+\mathcal{R}(g,\partial_ng,\phi,\partial_n\phi,\partial^2_n\phi,\partial_na_0)\\
  &=\frac{1}{4\det g}(\partial^2_n\det g)-\frac{1}{4(\partial_n\phi)^2}\partial^2_ng^{\alpha\beta}\partial_\alpha\phi\partial_\beta\phi+\mathcal{R}(g,\partial_ng,\phi,\partial_n\phi,\partial^2_n\phi,\partial_na_0)\\
  &=-\frac{1}{4}\frac{1}{\det g}\partial^2_n((\det g)g^{\alpha\beta})\xi_\alpha\xi_\beta+\mathcal{R}(g,\partial_ng,\phi,\partial_n\phi,\partial^2_n\phi,\partial_n a_0);
  \end{split}
  \]
  here, the function $\mathcal{R}$ may vary from line to line, but it always only depends on $g$, $\partial_n g$, $\phi$, $\partial_n\phi$, $\partial^2_n\phi$, $\partial_n a_0$ restricted on $\partial M$. Therefore, the identity \eqref{identity_amplitude} for $j=1$ implies that
  \[
  \frac{1}{\det g}\partial^2_n((\det g) g^{\alpha\beta})=\frac{1}{\det \widetilde{g}}\partial^2_n((\det \widetilde{g}) \widetilde{g}^{\alpha\beta})\quad \text{on }\partial M.
  \]
  and consequently $\partial^2_ng=\partial^2_n\widetilde{g}$ on $\partial M$. Repeating this process, we can similarly establish $\partial_n^mg=\partial_n^m\widetilde{g}$ on $\partial M$ for $m\geq 3,4,\ldots$. The proof is complete.
\end{proof}

In terms of the form~\eqref{EqBdyNorm} of the metric in partial boundary normal coordinates, we have thus recovered the jet of $k\in\CI\bigl([0,\eps)_{x^n};\CI((0,T)\times\pa N;S^2 T^*((0,T)\times\pa N))\bigr)$ at $x^n=0$. Extending $k$ in an arbitrary smooth fashion to $x^n\in(-2\eps,\eps)$ (and reducing $\eps$ further if necessary) thus furnishes us with an extension of $g$ to a metric on a slightly larger open manifold
\begin{equation}
\label{EqMfdLarge}
  M_1:=(0,T)\times N_1,\qquad N\subset N_1^{\rm int};
\end{equation}
we denote the extended metric by the same letter $g$, and we may assume that the first coordinate $\ft$ is a time function on $M_1\setminus M^{\rm int}$, i.e.\ its differential is timelike. In the notation used before the statement of Theorem~\ref{bdrydetermination}, the a priori discontinuous Lorentzian metric defined as $(\wt\Psi\circ\Psi^{-1})^*\wt g$ on $M$ and $g$ on $M_1\setminus M$ is, in fact, smooth as a consequence of Theorem~\ref{bdrydetermination}. We define
\begin{equation}
\label{EqSleeve}
  \cU_1 := (0,T) \times (N_1 \setminus N),
\end{equation}
which thus is an open set which completely surrounds $(0,T)\times\pa N$. See Figure~\ref{FigSleeve}.

\begin{figure}[!ht]
\centering
\includegraphics{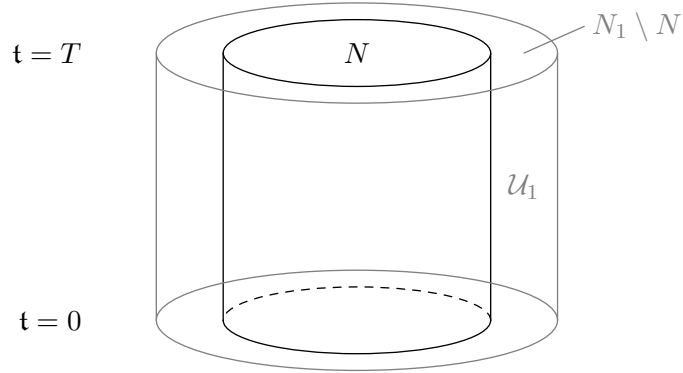}
\caption{The enlarged manifold $(0,T)\times N_1$, the original manifold $(0,T)\times N$, and the `sleeve' $\cU_1$ surrounding $(0,T)\times N$.}
\label{FigSleeve}
\end{figure}

\section{Interior determination}\label{interiordetermination}

In this section, we prove Theorem~\ref{ThmI}, except for the characterization of conformal factors $\beta$, in a sequence of steps.
\begin{enumerate}
\item In~\S\ref{SsGeo}, we describe null-geometric aspects of the extended manifold, and in~\S\ref{SsDist} we construct distorted plane waves.
\item In~\S\ref{SsScat}, we use a scattering control method to suppress reflections of distorted plane waves at $\pa M$.
\item In~\S\ref{SsInter}, we give a description of the strongest singularity produced by the nonlinear interaction of four distorted plane waves.
\item In~\S\ref{SsConf}, we use thus produced artificial point sources in $M^{\rm int}$ to reconstruct the conformal structure of the set $\mathbb{U}_g$ defined in~\eqref{EqDiamond}.
\item In~\S\ref{SsFact}, we use principal symbol arguments to determine the conformal class of $(g,a)$.
\end{enumerate}

We work with $M_1$ and $\cU_1$ defined in~\eqref{EqMfdLarge}--\eqref{EqSleeve}. For $p\in M_1$, denote the set of light-like tangent vectors at $p$ by
\[
L_p M_1=\{\zeta\in T_p M_1\setminus\{0\}\colon g(\zeta,\zeta)=0\}.
\]
The set of light-like covectors at $p$ is denoted by $L^{*}_p M_1$. The sets of future and past light-like vectors (covectors) are denoted by $L_p^+ M_1$ and $L_p^- M_1$ ($L^{*,+}_p M_1$ and $L^{*,-}_p M_1$).
Define the future directed light-cone emanating from $p$ by
\[
  \mathcal{L}^+(p)=\{\gamma_{p,\zeta}(t)\in  M_1\colon\zeta\in L^+_p M_1,\ t\geq 0\}\subset  M_1,
\]
where $\gamma_{p,\zeta}(t)$ is the null-geodesic with $\gamma_{p,\zeta}(0)=p$, $\dot\gamma_{p,\zeta}(0)=\zeta$; the parameter $t$ lies in the largest connected open interval $(t^-_{p,\zeta},t^+_{p,\zeta})$ containing $0$ for which $\gamma_{p,\zeta}$ is defined.

\subsection{Geometry of the extended manifold}
\label{SsGeo}

We will construct distorted plane waves whose singular support is concentrated near a suitably cut off geodesic inside of $M_1$. Since we are only assuming control on smooth geodesics inside $M$, we cannot expect to have good geometric properties of null-geodesics once they leave $M$ and enter $M_1$; for instance, they could turn back and reenter $M$, which could lead to conjugate points and other complications. Thus, we introduce a number of notions designed to capture those null-geodesics which enter and exit (and then possibly again enter) $M$, and the times at which they do so. We begin by defining the set
\[
  \wt E_{\rm i} := \{ (z_0,\zeta_0) \colon z_0 \in M_1\setminus M^{\rm int},\ \zeta_0\in L^+_{z_0}M_1,\ \exists\,t\in(0,t^+_{z_0,\zeta_0})\ \text{s.t.}\ \gamma_{z_0,\zeta_0}(t)\in(0,T)\times\pa N \}
\]
containing all initial data of null-geodesics $\gamma_{z_0,\zeta_0}$ which intersect the boundary $(0,T)\times\pa N$ for \emph{positive} affine parameter values. For $(z_0,\zeta_0)\in \wt E_{\rm i}$, we record with
\begin{equation}
\label{EqT0}
  \bft_0(z_0,\zeta_0) := \inf \{ t\in(0,t^+_{z_0,\zeta_0}) \colon \gamma_{z_0,\zeta_0}(t)\in(0,T)\times\pa N \} \in (0,t^+_{z_0,\zeta_0})
\end{equation}
the first positive time of contact with $(0,T)\times\pa N$. (Note that for $(z_0,\zeta_0)\in\wt E_{\rm i}$ with $z_0\in(0,T)\times\pa N$ and $\zeta_0$ inward pointing, $\bft_0(z_0,\zeta_0)$ measures the time at which $\gamma_{z_0,\zeta_0}$ first \emph{exits} $(0,T)\times N$.) We next discard those $(z_0,\zeta_0)$ for which $\gamma_{z_0,\zeta_0}$ has a tangential first contact with $(0,T)\times\pa N$; this leaves us with the set
\[
  E_{\rm i} := \{ (z_0,\zeta_0)\in \wt E_{\rm i} \colon \dot\gamma_{z_0,\zeta_0}(\bft_0(z_0,\zeta_0)) \notin T(\pa M) \},
\]
the subscript `i' standing for `in'. We remark that the set $E_{\rm i}^{\rm int}$ containing initial data \emph{outside} $M$,
\begin{equation}
\label{EqI}
  E_{\rm i}^{\rm int}:= E_{\rm i}\cap L^+(M_1\setminus M),
\end{equation}
is determined entirely by the restriction of $g$ to $M_1\setminus M^{\rm int}$. Note then that for $(z_0,\zeta_0)\in E_{\rm i}^{\rm int}$, the tangent vector $\dot\gamma_{z_0,\zeta_0}(\bft_0(z_0,\zeta_0))$ is necessarily pointing into $M^{\rm int}$. Thus, one can inquire whether $\gamma_{z_0,\zeta_0}$ leaves $M^{\rm int}$ again, which is equivalent to membership in $E_{\rm i}$ of the position and velocity when striking $(0,T)\times\pa N$; we thus set
\begin{equation}
\label{EqIO}
  E_{\rm io}^{\rm int} := \{ (z_0,\zeta_0) \in E_{\rm i}^{\rm int} \colon (\gamma_{z_0,\zeta_0}(t),\dot\gamma_{z_0,\zeta_0}(t))|_{t=\bft_0(z_0,\zeta_0)}\in E_{\rm i} \},
\end{equation}
with `io' standing for `in, out'; and for $(z_0,\zeta_0)\in E_{\rm io}^{\rm int}$, we record the exit time after the first entry as
\begin{equation}
\label{EqT1}
  \bft_1(z_0,\zeta_0) := \bft_0(z_0,\zeta_0) + \bft_0(\gamma_{z_0,\zeta_0}(t),\dot\gamma_{z_0,\zeta_0}(t))|_{t=\bft_0(z_0,\zeta_0)} \in (\mathbf{t_0}(z_0,\zeta_0),t^+_{z_0,\zeta_0}).
\end{equation}
Continuing, the null-geodesic $\gamma_{z_0,\zeta_0}(t)$ may, after exiting, reenter $(0,T)\times N$ yet again; since we do not want to keep track of effects \emph{inside} of $M$ which are due to poorly controlled geometry \emph{outside} of $M$, we will in our constructions below stop singularities of waves prior to reentering. For now, we merely define
\begin{equation}
\label{EqIOIT2}
\begin{split}
  E_{\rm ioi}^{\rm int} &:= \{ (z_0,\zeta_0) \in E_{\rm io}^{\rm int} \colon (\gamma_{z_0,\zeta_0}(t),\dot\gamma_{z_0,\zeta_0}(t))|_{t=\bft_1(z_0,\zeta_0)}\in E_{\rm i} \}, \\
  \bft_2(z_0,\zeta_0) &:= \bft_1(z_0,\zeta_0) + \bft_0(\gamma_{z_0,\zeta_0}(t),\dot\gamma_{z_0,\zeta_0}(t))|_{t=\bft_1(z_0,\zeta_0)} \in (\mathbf{t_0}(z_0,\zeta_0),t^+_{z_0,\zeta_0}),\quad (z_0,\zeta_0)\in E_{\rm ioi}^{\rm int}.
\end{split}
\end{equation}

In summary, we record the first entrance and exit time of $\gamma_{z_0,\zeta_0}$ into and out of $(0,T)\times N$, if they exist, and the first reentry time if applicable. See Figure~\ref{FigT}.

\begin{figure}[!ht]
\centering
\includegraphics{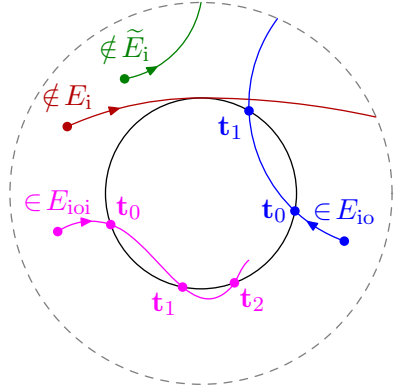}
\caption{Illustration of the definitions~\eqref{EqT0}--\eqref{EqIOIT2}, projected to the spatial mani\-fold $N_1$. The boundary of $N$ is the inner, solid black circle; the boundary of $N_1$ is the gray dashed outer circle.}
\label{FigT}
\end{figure}

We shall see in~\S\ref{SsScat} how to constructively determine the sets $E_{\rm io}^{\rm int}$ and $E_{\rm ioi}^{\rm int}$ from the data of the inverse problem.

\subsection{Construction of distorted plane waves}
\label{SsDist}

Distorted plane waves have singularities conormal to a submanifold of $M_1$ and can be viewed as Lagrangian distributions. We review the concept of Lagrangian distributions briefly, closely following the notation used in \cite{lassas2018inverse}. Recall that $T^*M_1$ is a symplectic manifold with canonical 2-form, given in local coordinates by $\omega=\sum_{j=1}^4\mathrm{d}\xi_j\wedge\mathrm{d}x^j$. A submanifold $\Lambda\subset T^*M_1$ is called Lagrangian if $n:=\mathrm{dim}\, \Lambda=4$ and $\omega$ vanishes on $\Lambda$. For $K$ a smooth submanifold of $M_1$, its conormal bundle
\[
N^*K=\{(x,\zeta)\in T^*M_1: x\in K,\, \langle \zeta,\theta\rangle=0, \,\theta\in T_xK\}
\]
is a Lagrangian submanifold of $T^*M_1$. Let $\Lambda$ be a smooth conic Lagrangian submanifold of $T^*\widetilde M_1\setminus 0$. We denote by $\mathcal{I}^\mu(\Lambda)$ the space of Lagrangian distributions of order $\mu$ associated with $\Lambda$. If $\Lambda=N^*K$ for some submanifold $K\subset M_1$, then $\mathcal{I}^\mu(K):=\mathcal{I}^\mu(N^*K)$ denotes the space of conormal distributions to $K$. For $u\in\mathcal{I}^\mu(\Lambda)$, one can define the principal symbol $\sigma_p(u)$ of $u$ with
\[
\sigma_p(u)\in S^{\mu+\frac{n}{4}}(\Lambda,\Omega^{1/2}\otimes L)/S^{\mu+\frac{n}{4}-1}(\Lambda,\Omega^{1/2}\otimes L),
\]
where $\Omega^{1/2}$ is the half-density on $M_1$ and $L$ is the Keller--Maslov line bundle of $\Lambda$. We refer to \cite[Chapter 4]{duistermaat1996fourier} for the precise definition and more discussions.

Fix a Riemannian metric $g^+$ on $M_1$. For $z_0\in M_1\setminus M$, $\zeta_0\in L_{z_0}^+M_1$, and $s_0>0$, denote by
\[
\mathcal{W}(z_0,\zeta_0,a_0)=\{\eta\in L_{z_0}^+M_1\colon \|\eta-\zeta_0\|_{g^+}<a_0 \}
\]
a neighborhood of $\zeta_0$ inside the future light cone. Let
\begin{align*}
Y(z_0,\zeta_0,a_0,s_0)&=\{\gamma_{z_0,\eta}(s_0)\in M_1 \colon \eta\in\mathcal{W}(z_0,\zeta_0,a_0) \},\\
K(z_0,\zeta_0,a_0,s_0,\bar t)&=\{\gamma_{z_0,\eta}(t)\in M_1 \colon \eta\in \mathcal{W}(z_0,\zeta_0,a_0),\ t\in(s_0,\bar t)\}, \\
\Lambda(z_0,\zeta_0,a_0,s_0,\bar t)&=\{(\gamma_{z_0,\eta}(t),r\dot{\gamma}_{z_0,\eta}(t)^\flat)\in T^*M_1\colon \eta\in \mathcal{W}(z_0,\zeta_0,a_0),\ t\in (s_0,\bar t),\ r\in\R\},
\end{align*}
and
\[
  K(z_0,\zeta_0,a_0,s_0)=K(z_0,\zeta_0,a_0,s_0,t^+_{z_0,\eta}),\qquad
  \Lambda(z_0,\zeta_0,a_0,s_0)=\Lambda(z_0,\zeta_0,a_0,s_0,t^+_{z_0,\eta}).
\]
Here, we use the parameter $a_0>0$ (which will always be chosen to be very small) to control the aperture of the congruence of null-geodesics emanating from $z_0$, while $s_0>0$ (also small) fixes the offset $t=s_0$ at which we start recording the null-geodesic. The set $\Lambda(z_0,\zeta_0,a_0,s_0)$ is the Lagrangian manifold that is the flowout from $L^{*,+}M_1\cap N^*Y(z_0,\zeta_0,a_0,s_0)$ under the Hamiltonian vector field $H_p$ associated with $p(x,\xi)=g^{jk}(x)\xi_j\xi_k$. Moreover, $K(z_0,\zeta_0,a_0,s_0)$ is the projection $\pi\Lambda(z_0,\zeta_0,a_0,s_0)$ to the base; until the first conjugate point of $\gamma_{z_0,\eta}$, it is a smooth null hypersurface with conormal bundle equal to $\Lambda(z_0,\zeta_0,a_0,s_0)$.

Let now $(z_0,\zeta_0)\in E_{\rm i}^{\rm int}$, and consider a source
\begin{equation}
\label{Eqh0}
  h_0\in\mathcal{I}^{\mu+3/2}(M_1;N^*Y(z_0,\zeta_0,a_0,s_0))
\end{equation}
conormal to $Y$. We wish to define a forward solution of the equation $\Box_g u_0=h_0$ adapted to our needs. Roughly speaking, we want $u_0$ to be an approximate solution, $\Box_g u_0-h_0\in\CI$, until shortly before the time when $\gamma_{z_0,\zeta_0}$ reenters $(0,T)\times N$ after exiting for the first time; the plan is to manufacture $u_0$ to be singular in a small neighborhood of a \emph{single} null-geodesic segment.

To accomplish this, note that the first intersection of $\gamma_{z_0,\zeta_0}$ with $(0,T)\times\pa N$ is transversal; this is an open condition in the initial conditions of the null-geodesics, hence we conclude that $(z_0,\eta)\in E_{\rm i}^{\rm int}$ for all $\eta\in\cW(z_0,\zeta_0,2 a_0)$ provided $a_0$ is sufficiently small. There are now three cases to consider:
\begin{enumerate}
\item\label{ItInotO} there exists $s_0>0$ so that $\cW(z_0,\zeta_0,2 a_0)\cap E_{\rm io}^{\rm int}=\emptyset$; that is, the curves $\gamma_{z_0,\eta}$ issuing from $z_0$ enter $(0,T)\times N$ but do not exit before time $\ft=T$. Or:
\item\label{ItIO} one has $(z_0,\zeta_0)\in E_{\rm io}^{\rm int}$. The same then holds true for all $\eta\in\cW(z_0,\zeta_0,2 a_0)$ when $s_0>0$ is sufficiently small; that is, the curves $\gamma_{z_0,\eta}$ enter and exit $(0,T)\times N$ (and might return, to be discussed momentarily). Or:
\item\label{ItIOnotI} $(z_0,\zeta_0)\notin E_{\rm io}^{\rm int}$, but for all $s_0>0$ we have $\cW(z_0,\zeta_0,2 a_0)\cap E_{\rm io}^{\rm int}\neq\emptyset$. This happens if $\gamma_{z_0,\zeta_0}$ exits $(0,T)\times N$ precisely at $\ft=T$. The null convexity assumption on $\pa M$ implies that $\gamma_{z_0,\zeta_0}$ exits transversally and thus does not return for a positive amount of time; but this means that we can choose $s_0>0$ so small that $\cW(z_0,\zeta_0,2 a_0)\cap E_{\rm ioi}^{\rm int}=\emptyset$.
\end{enumerate}

In cases \eqref{ItInotO} and \eqref{ItIOnotI}, we simply take $u_0:=\cQ_g h_0$, where $\cQ_g$ is a parametrix for $\Box_g$ microlocally near $\Lambda(z_0,\zeta_0,2 a_0,s_0)\times N^*Y(z_0,\zeta_0,2 a_0,s_0)$; more precisely, $\cQ_g$ is a distinguished parametrix in the terminology of \cite{DuistermaatHormanderFIO2} corresponding to propagation of singularities in the forward direction in time.\footnote{Alternatively, one can use a doubling construction to create a closed extension $N_1$ of $N$ so that $((0,T)\times N_1,g)$ is globally hyperbolic, and take $\cQ_g$ to be the forward fundamental solution.} By \cite{melrose1979lagrangian} we have $\cQ_g\in \mathcal{I}^{-3/2}(\Lambda_g\setminus N^*\mathrm{Diag})$, where $\mathrm{Diag}=\{(x,y)\in M_1\times M_1,\, x=y\}$, and $\Lambda_g\subset T^*M_1\times T^*M_1$ is given by
\[
\Lambda_g=\{(x,\xi,y,-\eta);\, g^{jk}(x)\xi_j\xi_k=0,\,(y,\eta)\in\Theta_{x,\xi} \},
\]
where $\Theta_{x,\xi}\subset T^*M_1$ is the bicharacteristic of $\square_g$ containing $(x,\xi)$.
 Thus, in this case,
\begin{equation}
\label{EqQ1}
  u_0=\mathcal{Q}_g(h_0),\quad u_0|_{M_1\setminus Y}\in \mathcal{I}^{\mu}(M_1\setminus Y;\, \Lambda(z_0,\zeta_0,a_0,s_0))
\end{equation}
is a conormal distribution whose singularities propagate along $\Lambda(z_0,\zeta_0,a_0,s_0)$ (cf. \cite[Lemma 3.1]{kurylev2018inverse}). We note that $u_0$ restricted to $(0,T)\times N$ satisfies $\Box_g u_0\in\CI((0,T)\times N)$.

It remains to consider case~\eqref{ItIO}. Here we have $\bft_1(z_0,\eta)<t^+_{z_0,\eta}$ for all $\eta\in\cW(z_0,\zeta_0,2 a_0)$. Due to the lower semicontinuity of the maximal existence time for ODEs, we may choose $\delta>0$ and $a_0>0$ so that
\[
  \bft_1(z_0,\eta) < \bft_1(z_0,\zeta_0)+\delta,\qquad \eta\in\cW(z_0,\zeta_0,2 a_0).
\]
In addition, in the case that $(z_0,\zeta_0)\in E_{\rm ioi}^{\rm int}$, i.e.\ the reentry time $\bft_2(z_0,\zeta_0)$ (which is lower semicontinuous) is well-defined, we may shrink $a_0>0$ and choose $\delta>0$ so small that
\[
  \bft_1(z_0,\eta) < \bft_1(z_0,\zeta_0)+\delta < \bft_1(z_0,\zeta_0)+3\delta < \bft_2(z_0,\eta),\qquad \eta\in\cW(z_0,\zeta_0,2 a_0).
\]
Let $\cQ'_g$ denote a microlocal parametrix for the forward problem for $\Box_g$ near
\[
  \Lambda(z_0,\zeta_0,2 a_0,s_0,\bft_1(z_0,\zeta_0)+3\delta)\times N^*Y(z_0,\zeta_0,2 s_0)
\]
so that $\Box_g(\cQ'_g h_0)-h_0$ is smooth near $\pi(\Lambda(z_0,\zeta_0,2 a_0,s_0,\bft_1(z_0,\zeta_0)+3\delta))$, where $\pi\colon T^*M_1\to M_1$ is the projection map. In order to avoid the reentry of singularities, we now choose a cutoff function $\chi\in\CIc(M_1)$ with
\begin{align*}
  \supp\chi \subset\ &\pi\bigl(\Lambda(z_0,\zeta_0,2 a_0,s_0/2,\bft_1(z_0,\zeta_0)+3\delta)\bigr), \\
  \chi\equiv 1\text{ on } &\pi\bigl(\Lambda(z_0,\zeta_0,a_0,s_0,\bft_1(z_0,\zeta_0)+2\delta)\setminus\Lambda(z_0,\zeta_0,a_0,s_0/2,s_0)\bigr),
\end{align*}
and put
\begin{equation}
\label{EqQ2}
  \cQ_g(h_0) := \chi\cQ'_g(h_0).
\end{equation}
This solves $\Box_g(\cQ_g(h_0))-h_0-[\Box_g,\chi]\cQ'_g(h_0)\in\CI(M_1)$, which is smooth in $(0,T)\times N$; moreover, $\cQ_g(h_0)$ itself is conormal of order $\mu$ along $K(z_0,\zeta_0,s_0)\cap M$ with principal symbol over $M$ equal to the principal symbol of $\cQ'_g$ times that of $h_0$ (see~\S\ref{SsInter} for details).

\begin{remark}
  The operator $\cQ_g$ in~\eqref{EqQ1} and \eqref{EqQ2} depends on various choices which are not recorded in the notation. However, in our application, only simple microlocal properties of $\cQ_g(h_0)$ matter, and these are independent of choices. Alternatively, one can partition phase space $T^*M_1$ into different regions which roughly correspond to the separation into the cases~\eqref{ItInotO}--\eqref{ItIOnotI}, and define a single operator $\cQ_g$ using a partition of unity; we leave the details to the reader.
\end{remark}

\subsection{Scattering control}
\label{SsScat}

Fix now $(z_0,\zeta_0)\in E_{\rm i}^{\rm int}$, $s_0>0$; let $h_0$ be as in~\eqref{Eqh0} and pick it to be a classical conormal distribution with connected wave front set, and define $u_0=\cQ_g h_0$ by~\eqref{EqQ1} or \eqref{EqQ2}, as appropriate. Let us assume that either $(z_0,\zeta_0)\notin E_{\rm io}^{\rm int}$, or otherwise that $\gamma_{z_0,\zeta_0}(t)$ has no conjugate point for $t\in[0,\bft_1(z_0,\zeta_0)+\delta]$ for some $\delta>0$; we discuss this condition, and how to arrange it, after the proof of Proposition~\ref{PropScat} below.

Define now
\begin{align*}
\cU(z_0,\zeta_0,a_0)&=\{\gamma_{z_0,\eta}(\bft_0(z_0,\eta)),\dot\gamma_{z_0,\eta}(\bft_0(z_0,\eta)) \colon \eta\in\mathcal{W}(z_0,\zeta_0,a_0)\}, \\
U(z_0,\zeta_0,a_0)&=\pi(\cU(z_0,\zeta_0,a_0))\subset\pa M,
\end{align*}
and denote by
\[
  \wt U(z_0,\zeta_0,a_0)\subset\pa M
\]
the open $a_0$-neighborhood of $U(z_0,\zeta_0,a_0)$ with respect to the Riemannian metric $g^+$. Note that the singularities of the restriction $u_0|_{\wt U(z_0,\zeta_0,2 a_0)}$ to the slightly larger set $\wt U(z_0,\zeta_0,2 a_0)$ are in fact contained in the smaller set $\wt U(z_0,\zeta_0,a_0)$; thus, we can choose $\wh f_0$ which has compact support in $\wt U(z_0,\zeta_0,2 a_0)$ and so that $\wh f_0-u_0|_{\wt U}$ is smooth in $\wt U(z_0,\zeta_0,2 a_0)$. Extend $\wh f_0$ by $0$ to $\pa M\setminus\wt U(z_0,\zeta_0,2 a_0)$. Importantly, since the metric $g$ on $M_1\setminus M$ is already known, we can use the value of $u_0$ on $M_1\setminus M$ to determine the symbol of $\wh f_0$ on $(0,T)\times\partial N$ by continuity. Denote by $\wh v_0$ to be the solution to the initial-boundary value problem
\begin{alignat}{2}
  \Box_g \wh v_0(x)&=0&\quad&\text{on }M,\nonumber\\
\label{linear_eq_v0}
  \wh v_0(x)&=\wh f_0(x)&\quad& \text{on }\partial M,\\
  \wh v_0(\ft,x')&=0, &\quad& \ft<0.\nonumber
\end{alignat}

We note that the singularities of the solution $\wh v_0$ might encounter several reflections at $\partial M$. However, before the first reflection we have $\wh v_0=u_0$ mod $C^\infty$. The goal of this section is the construction of additional boundary sources that remove multiple reflections. In order to be useful for our inverse problem, this needs to be accomplished based on the knowledge of $\Lambda_g^{\rm lin}$, but without knowing the metric $g$ itself inside $M$. The procedure is close in spirit to the scattering control introduced in \cite{caday2019scattering}.

\begin{prop}
\label{PropScat}
  Let $u_0$ and $\wh f_0$, depending on the choice of the aperture parameter $a_0>0$, be as above. There exists a constructive procedure, which only uses the Dirichlet-to-Neumann map $\Lambda_g^{\rm lin}$ and the knowledge of $g$ on $M_1\setminus M^{\rm int}$, for the determination of a value $a_0>0$ and the construction of a boundary source $\wh f_1$ on $(0,T)\times\pa N$ with the property that the solution of the equation
  \begin{alignat}{2}
    \Box_g v_0(x)&=0&\quad&\text{on }M,\nonumber\\
  \label{linear_eq_u0}
    v_0(x)&=f_0(x):=\wh f_0(x)+\wh f_1(x)&\quad& \text{on }\partial M,\\
    v_0(\ft,x')&=0, &\quad& \ft<0,\nonumber
  \end{alignat}
  satisfies $v_0-u_0\in\CI((0,T)\times N)$. That is, $v_0$ does not undergo any reflection at $\partial M$.
\end{prop}
\begin{proof}
  If $(z_0,\eta)\notin E_{\rm io}^{\rm int}$ for all $\eta\in\cW(z_0,\zeta_0,a_0)$, there is nothing to do: we automatically have $\wh v_0=u_0$ modulo $\CI((0,T)\times N)$ and can thus take $\wh f_1\equiv 0$. We discuss how to determine whether or not $(z_0,\eta)\notin E_{\rm io}^{\rm int}$ holds at the end of the proof.
  
  Otherwise, i.e.\ if there exists $\eta\in\cW(z_0,\zeta_0,a_0)$ so that $(z_0,\eta)\in E_{\rm io}^{\rm int}$, let
  \begin{equation}
  \label{EqFirstExit}
  \begin{split}
    \cV(z_0,\zeta_0,a_0)&=\{\gamma(\bft_1(z_0,\eta)),\dot{\gamma}(\bft_1(z_0,\eta)) \colon \eta\in\cW(z_0,\zeta_0,a_0),\ (z_0,\eta)\in E_{\rm io}^{\rm int}\}, \\
    V(z_0,\zeta_0,a_0) &:= \pi(\cV(z_0,\zeta_0,a_0)).
  \end{split}
  \end{equation}
  Since $\bft_1(z_0,\zeta_0)>\bft_0(z_0,\zeta_0)$, we can take $a_0$ small enough such that $\overline{\wt U(z_0,\zeta_0,2 a_0)}\cap V(z_0,\zeta_0,2 a_0)=\emptyset$. We return to the determination of the set $V(z_0,\zeta_0,a_0)$ of \emph{first} exit points of $\gamma_{z_0,\eta}$ at the end of the proof.
  
  Since the intersection of $K(z_0,\zeta_0,a_0)$ with $\pa M$ at $V(z_0,\zeta_0,a_0)$ is transversal, the solution $\wh v_0$ can be written as the sum of two \emph{conormal} distributions associated to the incoming null hypersurface $K(z_0,\zeta_0,a_0)$ (up until $V(z_0,\zeta_0,a_0)$) on the one hand, and the reflected null hypersurface on the other hand. Thus, we can write $\wh v_0=\wh v_0^{\mathrm{inc}}+\wh v_0^{\mathrm{ref}}$, where
   \begin{equation}
   \begin{split}
   \wh v_0^{\mathrm{inc}}(x)&=(2\pi)^{-1}\int_\R e^{\mathrm{i}\phi^{\mathrm{inc}}(x)\theta} a^{\rm inc}(x,\theta)\mathrm{d}\theta,\\
   \wh v_0^{\mathrm{ref}}(x)&=(2\pi)^{-1}\int_\R e^{\mathrm{i}\phi^{\mathrm{ref}}(x)\theta} a^{\rm ref}(x,\theta)\mathrm{d}\theta,
    \end{split}
   \end{equation}
   for $x$ in a  neighborhood of $V$ inside of $M$, and $\theta\in\mathbb{R}$; here, since our source $h_0$ is a classical (one-step polyhomogeneous) conormal distribution, the amplitudes $a^{\rm inc}$ and $a^{\rm ref}$ are such as well, with asymptotic expansions as $|\theta|\to\infty$
   \begin{equation}
   \label{EqSymbols}
     a^\bullet(x,\theta) \sim \sum_{j\geq 0} a^\bullet_j(x,\theta),\qquad \bullet={\rm inc,\ ref},\qquad
     a_j \in S^{\mu+\frac12-j}.
   \end{equation}
   We can assume that $\phi^{\mathrm{inc}}|_V=\phi^{\mathrm{ref}}|_V$. We remark that $\dd(\phi^\bullet\vert_{\partial M})$ is timelike, since $V$ is necessarily spacelike, cf.\ \cite[Lemma~2.5]{hintz2017reconstruction}; by switching the sign of $\theta$, we may assume that $\dd(\phi^\bullet\vert_{\partial M})$ is \emph{future} timelike. Moreover, the phase functions $\phi^{\mathrm{inc}}$ and $\phi^{\mathrm{ref}}$ both satisfy the eikonal equations \eqref{eq_phase}, but their normal derivatives differ by a sign, that is,
   \[
     \pa_\nu\phi^{\mathrm{inc}}=-\pa_\nu\phi^{\mathrm{ref}}\ \text{on } V.
   \]
   Invoking the zero Dirichlet boundary condition $\wh v_0^{\mathrm{inc}}+\wh v_0^{\mathrm{ref}}=0$ on $V$, the amplitudes satisfy the matching conditions
   \begin{equation}
   \label{EqMatch}
     a_j^{\rm ref}|_V = -a_j^{\rm ref}|_V,\quad j=0,1,\ldots.
   \end{equation}
   The Neumann data of $\wh v_0$ are thus
   \begin{equation}
   \begin{split}
   \Lambda^{\rm lin}_g(\wh f_0)\vert_V&=\partial_\nu \wh v_0\vert_V\\
      &=(2\pi)^{-1}\int e^{\mathrm{i}\phi^{\mathrm{inc}}(x)\theta} \left(2\mathrm{i}\theta(\partial_\nu\phi^{\mathrm{inc}})a_0^{\mathrm{inc}}+\partial_\nu(a_0^{\mathrm{inc}}+a_0^{\mathrm{ref}})+2\mathrm{i}\theta(\partial_\nu\phi^{\mathrm{inc}})a_1^{\mathrm{inc}}+\cdots\right)\mathrm{d}\theta.
   \end{split}
   \end{equation}
Note that $\phi^\bullet\vert_{\partial M}$ in a neighborhood of $V$ can be assumed to be known (with certain freedom of choice as long as $\{(x,\theta\mathrm{d}(\phi^\bullet\vert_{\partial M})):\, \theta\in\mathbb{R},\,\phi^\bullet(x)=0\}$ is a parametrization of the Lagrangian submanifold associated with $\Lambda^{\rm lin}_g(\wh f_0)$ near $V$).
    Since the metric $g$ on $M_1\setminus M$ is known, the value of $\partial_\nu\phi^{\mathrm{inc}}\vert_V<0$ is known. Thus, from the principal symbol of $\Lambda^{\rm lin}_g(\wh f_0)$ on $V$,
   \begin{equation}
   \label{EqNeumannSymbol}
     \sigma_p(\Lambda^{\rm lin}_g(\wh f_0))=2\mathrm{i}\theta(\partial_\nu\phi^{\mathrm{inc}})a_0^{\mathrm{inc}},
   \end{equation}
   we can determine the value of $a_0^{\mathrm{inc}}$ on $V$.
   
   The crucial insight is then that from the knowledge of $a_0^{\rm inc}$ thus obtained from the Dirichlet-to-Neumann map, and from the full symbol of $\Lambda_g^{\rm lin}(\wh f_0)|_V$, we can reconstruct the full jet of all $a_j^\bullet$ with $(j,\bullet)\neq(0,{\rm inc})$. Indeed, recall the matching conditions~\eqref{EqMatch}, and note that the amplitudes satisfy the transport equations
   \begin{equation}\label{transport_eqs}
   \begin{split}
   \mathrm{i}T^{\mathrm{inc}}a_j^{\mathrm{inc}}&=-\Box_ga_{j-1}^{\mathrm{inc}},\\
    \mathrm{i}T^{\mathrm{ref}}a_j^{\mathrm{ref}}&=-\Box_ga_{j-1}^{\mathrm{ref}},
   \end{split}
   \end{equation}
   for $j\geq 1$, where the transport operators $T^{\mathrm{inc}}$ and  $T^{\mathrm{ref}}$ are defined in \eqref{transport_T} with $\phi=\phi^{\mathrm{inc}}$ and $\phi= \phi^{\mathrm{ref}}$. By solving these transport equations in the exterior domain $M_1\setminus M$, the amplitudes $a_0^{\mathrm{inc}}$ and $a_0^{\mathrm{ref}}$ in can be determined in $M_1\setminus M$. Consequently, $\partial_\nu(a_0^{\mathrm{inc}}+a_0^{\mathrm{ref}})\vert_V$ can be determined. Then from the subprincipal symbol of $\Lambda^{\rm lin}_g(\wh f_0)$, we can determine $a_1^{\mathrm{inc}}\vert_V$. Continuing with this process, we can determine $a_j\vert_V$ for all $j=0,1,2,\ldots$. Therefore, the boundary trace $\wh v_0^{\mathrm{ref}}=-\wh v_0^{\mathrm{inc}}\vert_V=-u_0\vert_V$ of the reflected distorted plane wave can be determined, modulo smooth functions, from $\Lambda_g^{\rm lin}(\wh f_0)$. But this means that we can construct the boundary source
  \[
    f_0 = \wh f_0 + \wh f_1,\qquad \wh f_1 := -\wh v_0^{\rm ref},
  \]
  which has the desired property. Carefully note that this uses that $u_0$ itself does not have artificial singularities introduced at later times; this is precisely due to our cutoff construction in~\eqref{EqQ2} for singularities which enter, then leave, and \emph{do not reenter $(0,T)\times N$} since we cut them off in time.

  It remains to tie up two loose ends from the beginning of the proof. Firstly, whether or not the condition $(z_0,\eta)\notin E_{\rm io}^{\rm int}$ holds for all $\eta\in\cW(z_0,\zeta_0,a_0)$ can be checked as follows. First, for sufficiently small $a_0>0$, the set of singularities of the Neumann data $\Lambda_g^{\rm lin}(\wh f_0)$ inside the set $\wt U(z_0,\zeta_0,a_0)$ is contained in the singular set of $\wh f_0$; this follows from the transversality of $\gamma_{z_0,\eta}$ at the first point of intersection with $(0,T)\times\pa N$. For such $a_0$, we thus rule out that the first reflection happens inside $\wt U(z_0,\zeta_0,a_0)$. Next, if $\Lambda_g^{\rm lin}(\wh f_0)$ has no singularities outside of $\wt U(z_0,\zeta_0,a_0)$, we can conclude that $(z_0,\eta)\notin E_{\rm io}^{\rm int}$; indeed, otherwise (as we are assuming that $\gamma_{z_0,\eta}(t)$ has no conjugate points near $[0,\bft_1(z_0,\eta)]$) the Neumann trace would have a singularity by~\eqref{EqNeumannSymbol}.

  Secondly, the determination of the first exit points, in the case that $(z_0,\eta)\in E_{\rm io}^{\rm int}$ for some $\eta\in\cW(z_0,\zeta_0,a_0)$ no matter how small $a_0>0$ is, can be accomplished as follows. For sufficiently small $a_0>0$, the singular set of $\Lambda_g^{\rm lin}(\wh f_0)$ will have at least one connected component disjoint from the component contained in $\wt U(z_0,\zeta_0,a_0)$; one of them is the place of first reflection. (This is the place where we take advantage of our choice of the source as having connected wave front set, thus singular set.) The key point is that if we perform the above scattering control construction near a component of the singular set of $\Lambda_g^{\rm lin}(\wh f_0)$ which is disjoint from $V(z_0,\zeta_0,a_0)$, the solution $v_0$ of the modified equation~\eqref{linear_eq_u0} will remain singular at $V(z_0,\zeta_0,a_0)$. That is, we can determine whether we performed scattering control at the correct, i.e.\ first, reflection point simply by checking if, after scattering control, the singular set of the Neumann data of the solution $v_0$ of~\eqref{linear_eq_u0} has a single connected component.\qedhere
\end{proof}

Regarding the assumption on the absence of conjugate points along $\gamma_{z_0,\zeta_0}$, note that for $(z_1,\zeta_1):=\gamma_{z_0,\zeta_0}(\bft_0(z_0,\zeta_0))\in(0,T)\times\pa N$, the null-geodesic $\gamma_{z_1,\zeta_1}$ has no conjugate points prior to hitting $\pa M$ at the point $\gamma_{z_0,\zeta_0}(\bft_1(z_0,\zeta_0))$. Thus, for some small $\delta>0$, $\gamma_{z_0,\zeta_0}(t)$ has no conjugate points for $t\in[\bft_0(z_0,\zeta_0)-\delta,\bft_1(z_0,\zeta_0)+\delta]$. Correspondingly, distorted plane waves in $M_1$ generated by conormal sources near $\gamma_{z_0,\zeta_0}(\bft_0(z_0,\zeta_0)-\delta)$ are conormal until a bit beyond the exit out of $(0,T)\times N$, and in particular the Neumann data have a conormal singularity near $\gamma_{z_0,\zeta_0}(\bft_1(z_0,\zeta_0))$; the singularity in turn is the zero level set of the restriction to $\pa M$ of a phase function $\phi^{\rm inc}$ which extends, as a solution of the eikonal equation, \emph{smoothly} to a neighborhood of $\gamma_{z_0,\zeta_0}(\bft_1(z_0,\zeta_0))$ inside $M_1\setminus M^{\rm int}$.

Thus, given $(z_0,\zeta_0)$, we can determine $\delta_0>0$ so that for all $\delta\in(0,\delta_0)$, the phase function $\phi^{\rm inc}$ is smooth in a $\delta_0$-neighborhood of $\gamma_{z_0,\zeta_0}(\bft_1(z_0,\zeta_0))$. Moreover, $\delta_0$ can be chosen to depend continuously on $(z_0,\zeta_0)$, $\|\zeta\|_{g^+}=1$. Thus, we can (constructively) take the extension $N_1\setminus N$ in~\eqref{EqSleeve} sufficiently small so that the absence of conjugate points holds for all starting points $z_0\in M_1\setminus M^{\rm int}$.

\subsection{Nonlinear interactions of distorted plane waves}
\label{SsInter}

For $j=1,2,3,4$, take $z_j\in M_1\setminus M$ and $\zeta_j\in L^+_{z_j}M_1$ so that $(z_j,\zeta_j)\in E_{\rm i}^{\rm int}$; we recall that this condition can be checked only using the knowledge of $g|_{M_1\setminus M^{\rm int}}$, which we have already constructed. Choosing $a_0,s_0>0$ sufficiently small, and taking sources
\begin{equation}
\label{Eqhj}
  h_j\in\cI^{\mu+3/2}(M_1;N^*Y(z_j,\zeta_j,a_0,s_0))
\end{equation}
which we take to be classical conormal with connected wave front set, we can then use Proposition~\ref{PropScat} to produce Dirichlet data $f_j$ such that the solution $v_j$ to the equation
\begin{equation}
\label{Eqvj}
\begin{alignedat}{2}
  \Box_g v_j(x) &= 0, &\quad & \text{on }M,\\
  v_j(x) &= f_j(x), &\quad& \text{on }\partial M,\\
  v_j(\ft,x') &= 0, &\quad& \ft<0,
\end{alignedat}
\end{equation}
is the restriction to $M$ of a conormal distribution in $\cI^\mu(M_1;\, \Lambda(z_j,\zeta_j,a_0,s_0))$ which agrees with $u_j:=\cQ_g(h_j)$ on $(0,T)\times N$.

As Dirichlet data for four interacting distorted plane waves, we then take
\begin{equation}\label{boundarysource}
f=\sum_{j=1}^4\epsilon_j f_j;
\end{equation}
denote by $v=\sum_{j=1}^4 \epsilon_jv_j$ the solution of the linear initial-boundary value problem $\Box_g v=f$, $v|_{\pa M}=f$, $v|_{\ft<0}=0$. Let $u$ denote the solution of the \emph{nonlinear} equation~\eqref{maineq} with $f=\sum_{j=1}^4\epsilon_j f_j$. Writing $w=Q_g(F)$ if $w$ solves the linear wave equation
\[
\begin{split}
\Box_g w(x)&=F,\quad\text{on }M,\\
w(x)&=0,\quad \text{on }\partial M,\\
w&=0, \quad \ft<0,
\end{split}
\]
the wave produced by the four-fold intersection is then
\begin{equation}
\label{EqU4}
\mathcal{U}^{(4)}:=\partial_{\epsilon_1}\partial_{\epsilon_2}\partial_{\epsilon_3}\partial_{\epsilon_3}u\vert_{\epsilon_1=\epsilon_2=\epsilon_3=\epsilon_4=0}=Q_g(av_1v_2v_3v_4).
\end{equation}

We proceed to describe the microlocal structure of $\cU^{(4)}$ under certain nondegeneracy assumptions. Let us abbreviate
\[
K_j=K(z_j,\zeta_j,a_0,s_0,\bar t_j),\quad\quad \Lambda_j=\Lambda(z_j,\zeta_j,a_0,s_0,\bar t_j),
\]
where $\bar t_j$ is fixed to be larger than $\bft_1(z_j,\zeta_j)$ if the latter is defined and equal to $t_{z_j,\zeta_j}^+$ otherwise, and $\bar t_j<\bft_2(z_j,\zeta_j)$ if defined; and $a_0,s_0>0$ are chosen small. Assume that 
\begin{enumerate}
\item $K_i,\,K_j$, $i\neq j$, intersect transversally at a co-dimension $2$ submanifold $K_{ij}\subset M^{\rm int}$;
\item $K_i,\, K_j,\, K_k$, $i,j,k$ distinct, intersect transversally at a co-dimension $3$ submanifold $K_{ijk}\subset M^{\rm int}$.
\end{enumerate}
Under the assumption that $z_j\notin\gamma_{z_i,\zeta_i}([0,t^+_{z_i,\zeta_i}))$ for $i\neq j$, these two conditions are automatically satisfied if we choose the aperture $a_0$ of the sources $h_j$, see~\eqref{Eqhj}, sufficiently small. Indeed, in the limit $a_0\to 0$, the sets $K_j$ tend to the null-geodesics $\gamma_{z_j,\zeta_j}$, the tangent vectors of which are equal to the normal vector of $K_j$ at points on $\gamma_{z_j,\zeta_j}$; but since the $z_j$ are causally unrelated, any two $\gamma_{z_j,\zeta_j}$ necessarily intersect transversally, and a third null-geodesic $\gamma_{z_k,\zeta_k}$ has, at a triple intersection point $\gamma_{z_i,\zeta_i}(t_i)=\gamma_{z_j,\zeta_j}(t_j)=\gamma_{z_k,\zeta_k}(t_k)$, a tangent vector not collinear to either of the other two. Indeed, the only lightlike tangent vectors in the span of $\dot\gamma_{z_i,\zeta_i}(t_i)$ and $\dot\gamma_{z_j,\zeta_j}(t_j)$ (which is a 2-dimensional linear Lorentzian vector space) are scalar multiples of $\dot\gamma_{z_i,\zeta_i}(t_i)$ and $\dot\gamma_{z_j,\zeta_j}(t_j)$; this implies the transversality of $K_i$, $K_j$, $K_k$ for sufficiently small $s_0>0$.

In order to proceed, let us make the additional assumption that
\begin{equation}
\label{Eq1234}
\begin{split}
  &\text{for all $s_0>0$, }\,K_1,\,K_2,\,K_3,\,K_4\,\ \text{intersect transversally at a point $q_0\in M^{\rm int}$}, \\
  &\qquad\qquad\text{prior to exiting $(0,T)\times N$ for the first time (if they exit at all).}
\end{split}
\end{equation}
We discuss the verification of this assumption using the already reconstructed data at the end of this section; here, we note that due to our assumption on the absence of conjugate points in $M$, there is at most a single such intersection point $q_0$. We then put
\[
 \Lambda_{ij}=N^*K_{ij},\quad \Lambda_{ijk}=N^*K_{ijk},\qquad
 \Lambda_{q_0}=T^*_{q_0}M\setminus 0;
\]
these are all Lagrangian submanifolds in $T^*M_1$. For any $\Gamma\subset T^*M_1$, we denote by $\Gamma^g$ the flow-out of $\Gamma\cap L^{*,+} M_1$ under the geodesic flow of $g$ lifted to $T^*M_1$. Finally, set
\begin{gather*}
\Lambda^{(1)}=\bigcup_{i=1}^4\Lambda_i;\quad\,\,\Lambda^{(2)}=\bigcup_{i,j=1}^4\Lambda_{ij};\quad\,\,\Lambda^{(3)}=\bigcup_{i,j,k=1}^4\Lambda_{ijk};\\
K^{(1)}=\bigcup_{i=1}^4K_i;\quad K^{(2)}=\bigcup_{i,j=1}^4K_{ij};\quad K^{(3)}=\bigcup_{i,j,k=1}^4K_{ijk}; \\
\Xi=\Lambda^{(1)}\cup \Lambda^{(3),g}\cup \Lambda_{q_0}.
\end{gather*}

We then have
\begin{equation}
\label{EqU4incref}
  \mathcal{U}^{(4)}=\mathcal{U}^{\mathrm{inc}}+\mathcal{U}^{\mathrm{ref}},
\end{equation}
in $(0,T)\times N$, where (cf. \cite[Proposition 3.11]{lassas2018inverse})
\[
  \mathcal{U}^{\mathrm{inc}}=\mathcal{Q}_g(av_1v_2v_3v_4)\in \mathcal{I}^{4\mu+3/2}(\Lambda_{q_0}^g\setminus \Xi);
\]
here $\cQ_g$ denotes a forward parametrix for $\Box_g^{-1}$ on $M_1$ microlocally near $(\Lambda_{q_0}^g\setminus\Xi)\times\Lambda_{q_0}$ which is truncated after the passage of null-bicharacteristics out of $M$ into $M_1\setminus M$, analogously to the construction of~\eqref{EqQ2}.

In order to describe the symbol of $\cU^{\rm inc}$, we introduce some further notation: define the set of initial data for null-geodesics inside $(0,T)\times N^{\rm int}$ which strike $(0,T)\times\pa N$ by
\begin{equation}
\label{EqVisible}
  \cS := \{ (z_0,\zeta_0) \colon z_0\in(0,T)\times N^{\rm int},\ \zeta_0\in L^+_{z_0}M,\  \exists\,t\in(0,t^+_{z_0,\zeta_0})\ \text{s.t.}\ \gamma_{z_0,\zeta_0}(t)\in(0,T)\times\pa N \};
\end{equation}
importantly, the null convexity of $\pa M$ implies that $\gamma_{z_0,\zeta_0}$ in fact strikes $\pa M$ \emph{transversally} and thus exits $M$ for a positive amount of time. For $(z_0,\zeta_0)\in\cS$, we then define the exit time
\[
  \bft_0(z_0,\zeta_0) := \inf \{ t\in(0,t^+_{z_0,\zeta_0}) \colon \gamma_{z_0,\zeta_0}(t)\in(0,T)\times\pa N \}.
\]

If $(q_0,\zeta)\in\cS$ and $(y,\eta)=(\gamma_{q_0,\zeta}(\bft_0(q_0,\zeta)),\dot{\gamma}_{q_0,\zeta}(\bft_0(q_0,\zeta)^\flat)$, the wave $\cU^{\rm inc}$ incident to $(0,T)\times\pa N$ then has principal symbol given by (cf. \cite[Proposition 3.12]{lassas2018inverse})
\begin{equation}\label{symbol_Uinc}
\sigma_p(\mathcal{U}^{\mathrm{inc}})(y,\eta)=-24(2\pi)^{-3}a(q_0)\sigma_{p}(\mathcal{Q}_g)(y,\eta,q_0,\zeta^\flat)\prod_{j=1}^4\sigma_p(v_j)(q_0,\theta_j).
\end{equation}
Here $\theta_j\in N^*_{q_0}K_j$ denote the unique covectors for which $\zeta^\flat=\sum_{j=1}^4\theta_j$, and $\sigma_p(\mathcal{Q}_g)$ is the principal symbol of the FIO $\mathcal{Q}_g$ on $\Lambda_{q_0}^g$ away from $N^*\mathrm{Diag}$.

We can say more: in terms of~\eqref{EqU4incref}, we can write, near the point $y\in(0,T)\times\pa N$,
\begin{equation}
\begin{split}
\mathcal{U}^{\mathrm{inc}}(x)&=(2\pi)^{-1}\int e^{\mathrm{i}\phi^{\mathrm{inc}}(x)\theta}a^{\mathrm{inc}}(x,\theta)\mathrm{d}\theta\\
\mathcal{U}^{\mathrm{ref}}(x)&=(2\pi)^{-1}\int e^{\mathrm{i}\phi^{\mathrm{ref}}(x)\theta}a^{\mathrm{ref}}(x,\theta)\mathrm{d}\theta,
\end{split}
\end{equation}
where the phase functions $\phi^{\rm inc}$ and $\phi^{\rm ref}$ satisfy the eikonal equation~\eqref{eq_phase} and the boundary conditions $\phi^{\mathrm{inc}}=\phi^{\mathrm{ref}}$ and $\pa_\nu\phi^{\rm inc}=-\pa_\nu\phi^{\rm ref}$ on $\pa M$, while the symbols $a^{\bullet}$ are as in~\eqref{EqSymbols}. The Dirichlet boundary condition implies that $a^{\mathrm{ref}}(x,\theta)=-a^{\mathrm{inc}}(x,\theta)$ for $x\in\partial M$. Then
\begin{equation}
\label{EqpanU4}
\partial_\nu \mathcal{U}^{(4)}\vert_{\partial M}=(2\pi)^{-3}\int e^{\mathrm{i}\phi^{\rm inc}(x)\theta}\left(2\mathrm{i}\theta(\partial_\nu\phi^{\rm inc})a^{\mathrm{inc}}+\partial_\nu(a^{\mathrm{inc}}+a^{\mathrm{ref}})\right)\mathrm{d}\theta.
\end{equation}
Therefore, the contribution of the four-fold interaction to the Dirichlet-to-Neumann map is
\[
\begin{split}
&\sigma_p\bigr(\partial_{\epsilon_1}\partial_{\epsilon_2}\partial_{\epsilon_3}\partial_{\epsilon_4}\Lambda_{g,a}\left(\epsilon_1f_1+\epsilon_2f_2+\epsilon_3f_3+\epsilon_4f_4\right)\vert_{\epsilon_1=\epsilon_2=\epsilon_3=\epsilon_4=0}\bigr)(x,\theta)\\
&\quad=\sigma_p(\partial_\nu \mathcal{U}^{(4)}\vert_{\partial M})(x,\theta)\\
&\quad=-2\mathrm{i}\theta\xi_n a^{\mathrm{inc}}_0(x,\theta), \qquad
  \xi_n:=-\sqrt{-g^{\alpha\beta}(x)\xi'_\alpha\xi'_\beta},\quad \xi':=\dd(\phi^{\rm inc}|_{\pa M}).
\end{split}
\]
Thus, we can recover $a^{\mathrm{inc}}_0(x,\theta)$ for $x\in\partial M$ (and then the complete jet of $a^\bullet_j$, $j=0,1,2,\ldots$, and $\phi^\bullet$ at $\pa M$, where $\bullet=\rm inc,\,ref$; and therefore we can recover the (full) symbol of $\cU^{\rm inc}\vert_{\partial M}$ and then the symbol of $\cU^{\rm inc}$ in $M_1\setminus M^{int}$.

Thus, if the interaction point $q_0\in M^{\rm int}$ is such that the set of $\zeta\in L^+_{q_0}M$ with $(q_0,\zeta)\in\cS$ is nonempty, then this set is automatically open, and thus $\cU^{\rm inc}$ is singular on a codimension $1$ (i.e.\ $2$-dimensional) submanifold of $(0,T)\times\pa N$. If there does not exist $\zeta\in L^+_{q_0}M$ with $(q_0,\zeta)\in\cS$, then $\cU^{\rm inc}$ is smooth on $(0,T)\times\pa N$.

We argue that we can verify if assumption~\eqref{Eq1234} is satisfied: indeed, if for sufficiently small $a_0>0$ the null hypersurfaces $K_1,K_2,K_3,K_4$ do not intersect all at once, the four-fold interaction $\cU^{(4)}$ is smooth away from $\Xi=\Lambda^{(1)}\cup\Lambda^{(3),g}$. But the intersection of the projection of $\Xi$ to the base with $\pa M$ tends to a set of codimension $2$ (meaning: of Hausdorff dimension $3{-}2=1$); note here that the 1-dimensional Hausdorff measure of any triple intersection $K_{i j k}$ tends to $0$ as $a_0>0$, and for any $p\in K_{i j k}$, the dimension of the set of lightlike directions in $N_p^*K_{i j k}$ (i.e.\ the set of lightlike covectors which have unit length with respect to $g^+$) is equal to $2$ (compared to the dimension of the set of lightlike directions in $N_{q_0}^*M$, which is equal to $3$). See also \cite[\S3.3.2]{kurylev2018inverse}.

Similarly, if the intersection $\bigcap_{j=1}^4 K_j$ is nonempty but not transversal (but, due to the absence of conjugate points, still only consists of a single point), the wave front set of the product $v_1 v_2 v_3 v_4$ is contained in $\bigcup N^*K_{i j k}$, thus $\cU^{(4)}$ is singular only on a small set in the sense just explained.

Finally, if some $z_j$ did lie on the null-geodesic $\gamma_{z_i,\zeta_i}$, then we would have $K_j\subset K_i$. If the intersection $K_i\cap\bigcap_l K_l$, $l\neq i,j$, were transversal, we would have only a three-fold interaction which can be detected as above. Otherwise, we only have a two-fold interaction, which does not produce any additional singularities; this can thus be detected as well.

\subsection{Determination of the conformal class of the metric}
\label{SsConf}

We continue using the notation of the previous section; thus, we take $z_j\in M_1\setminus M$ and $\zeta_j\in L^+_{z_j}M_1$ for $j=1,2,3,4$ with $(z_j,\zeta_j)\in E_{\rm i}^{\rm int}$, and consider sources $h_j$ and solutions of the initial-boundary value problem $v_j$ with Dirichlet data $f_j$; we moreover take $f=\sum_{i=1}^4 \epsilon_i f_i$ as in~\eqref{boundarysource} and $u$ as in~\eqref{maineq}, and put $v_j=\frac{\partial}{\partial \epsilon_j}u\vert_{\epsilon_1=\epsilon_2=\epsilon_3=\epsilon_4=0}$, which solves the linear initial-boundary value problem with source $f_j$. Thus,
\[
\partial_\nu v_j\vert_{(0,T)\times\partial N}=\frac{\partial}{\partial \epsilon_j}\Lambda_{g,a}(f)\vert_{\epsilon_1=\epsilon_2=\epsilon_3=\epsilon_4=0}=\frac{\partial}{\partial \epsilon_j}\Lambda_{g,a}\biggl(\sum_{i=1}^4\epsilon_if_i\biggr)\bigg|_{\epsilon_1=\epsilon_2=\epsilon_3=\epsilon_4=0}.
\]

From now on, whenever $(z_j,\zeta_j)\in E_{\rm ioi}^{\rm int}$, we shall write $\wt\gamma_{z_j,\zeta_j}$ for the restriction of $\gamma_{z_j,\zeta_j}$ to parameters $t\in[0,\bft_2(z_j,\zeta_j))$, i.e.\ we stop the null-geodesic right before it would reenter $(0,T)\times\pa N$; otherwise, we let $\wt\gamma_{z_j,\zeta_j}=\gamma_{z_j,\zeta_j}$ with domain $[0,t^+_{z_j,\zeta_j})$. First, we test whether the intersection point $q_0$ of $\gamma_{z_j,\zeta_j}$ lies on the boundary; but $\bigcap_{j=1}^4\gamma_{z_j,\zeta_j}=q_0\in (0,T)\times \partial N$ if and only if $q_0\in\lim_{s_0\rightarrow 0}\bigcap_{j=1}^4\,\mathrm{singsupp}\,(\partial_\nu v_j\vert_{(0,T)\times\partial N})$.

As discussed in the previous section, we can determine whether the four-fold intersection of the null hypersurfaces $K_j$ is nonempty and transversal. Assuming thus that it is, then away from
\[
  \mathcal{K}^{(3)}=\pi(\Lambda^{(3),g})
\]
which bounds the set of spacetime points to which singularities arising from triple intersections could propagate, $\mathcal{U}^{\mathrm{inc}}$ is the restriction to $M$ of an element of $\mathcal{I}^{4\mu+3/2}(\Lambda_{q_0}^g\setminus(\Xi\cup T^*_{\cK^{(3)}}M_1))$; and as discussed after~\eqref{EqpanU4}, we can recover the full symbol of $\cU^{\rm inc}$ restricted to $(0,T)\times\pa N$. Upon taking the aperture $a_0$ of our distorted plane waves to $0$, the set
\[
  \{ (\ft,x) \in (0,T)\times \pa N \colon \cU^{\rm inc}\ \text{is smooth near $(0,\ft)\times\{x\}$ but not at $(\ft,x)$} \}
\]
tends to an open dense subset of the boundary light observation set
\[
  \cE_M(q_0) := ((0,T)\times\pa N) \cap \{ \gamma_{q_0,\zeta_0}(\bft_0(q_0,\zeta_0)) \colon \zeta_0\in L^+_{q_0}M \}
\]
(which, due to the `no conjugate points' hypothesis, is a smooth spacelike hypersurface), and thus we can recover $\cE_M(q_0)$.

Now, we claim that every point $q_0\in\mathbb{U}_g$ lying in $(0,T)\times N^{\rm int}$ is the transversal intersection of four distorted plane waves originating outside $(0,T)\times N$. To see this, we first prove the following result:

\begin{lemma}
\label{LemmaGeod}
  Given $q_0\in\mathbb{U}_g\cap((0,T)\times N^{\rm int})$, there exists a null-geodesic $\mu\colon[0,1]\to M$ with $\mu(0)\in(0,T)\times\pa N$ and $\mu(1)=q_0$, and which moreover has no cut points.
\end{lemma}
\begin{proof}
  By definition of $\mathbb{U}_g$, there exists $(\ft_0,x_0)\in(0,T)\times\pa N$ so that $q_0$ lies in the causal future of $(\ft_0,x_0)$; on the other hand, $q_0$ does \emph{not} lie in the causal future of $(T,x_0)$. Thus, we have
  \[
    \bar\ft := \sup \{ t\in(0,T) \colon q_0\in J_g^+((t,x_0)) \} \in [\ft_0,T).
  \]
  We claim that there exists a null-geodesic in $M$ joining $(\bar\ft,x_0)$ to $q_0$. We can select $t_j\in(0,T)$ and causal curves $\gamma_j\colon[t_j,\ft(q_0)]\to M$ so that $t_1\leq t_2\leq\ldots\nearrow\bar\ft$ and $\gamma_j(t_j)=(t_j,x_0)$, $\gamma_j(\ft(q_0))=q_0$; here, we parameterize the $\gamma_j$ to that $\ft\circ\gamma_j(t)=t$. Since $N$ is compact, the Arzel\`a--Ascoli theorem implies that upon passing to a subsequence of the $\gamma_j$, we may assume the existence of a limit $\lim_{j\to\infty}\gamma_j(s)=\gamma(s)$ for all $s\in[\bar\ft,\ft(q_0)]$; note here that $\gamma(\bar\ft)=(\bar\ft,x_0)$ and $\gamma(\ft(q_0))=q_0$, and $s\mapsto\gamma(s)$ is Lipschitz since each $\gamma_j$ is Lipschitz (with uniform Lipschitz constant). Let $s_0\in[\bar\ft,\ft(q_0)]$ be the supremum of all $s\in[\bar\ft,\ft(q_0))$ so that $\gamma(s)\in\pa M$; in particular, $s_0<\ft(q_0)$.
  
  We will use $\gamma|_{[s_0,\ft(q_0)]}$ to construct the desired null-geodesic $\mu$. Consider any $s\in(s_0,\ft(q_0)]$, and suppose that we have already re-defined $\gamma$ to be a smooth reparameterized null-geodesic on $[s,\ft(q_0)]$. (This is initially satisfied for $s=\ft(q_0)$.) Then $\gamma(s)\in(0,T)\times N^{\rm int}$. Denote by $\cU\subset M^{\rm int}$ a geodesically convex neighborhood of $q=\gamma(s)$ with closure $\bar\cU\subset M^{\rm int}$, and let $\eps>0$ be such that $\gamma_{q,\zeta}(1)\in\cU$ for all $\zeta\in T_q M$ with $|\zeta|_{g^+}\leq 3\eps$. Since $\gamma(s_0)\notin\cU$, there exist $s'\in(s_0,s)$ and $\zeta\in T_q M$ with $|\zeta|_{g^+}=\eps$ so that
  \[
    q':=\gamma(s')=\gamma_{q,\zeta}(1).
  \]
  By construction of $\gamma$, the tangent vector $\zeta$ is necessarily past causal.
  
  Suppose that $\zeta$ was past \emph{timelike}. For $\digamma>1$ (chosen momentarily), $j\in\mathbb N$, and $\delta>0$ small, define the curve\footnote{We work here in the product splitting $(0,T)\times N$ of $M$, and write for a point $p=(t,x)\in M$ and a real number $a\in(-t,T-t)$: $p+(a,0):=(t+a,x)$. A more intrinsic (but notationally less convenient) perspective which avoids using the product splitting proceeds by instead using the exponential map based at $p$ and evaluated at $a$ times a fixed future timelike vector on $M$ which is tangent to $\pa M$.}
  \[
    \tilde\gamma_{j,\delta}(t) := \gamma_j(t) + \bigl(\delta e^{\digamma(t-\ft(q'))},0\bigr),\quad t\in[t_j,\ft(q')];
  \]
  we claim that this curve is timelike. Indeed, using the form~\eqref{EqIMetric} of the metric, and writing $\gamma_j(t)=(t,x'_j(t))$, the causal nature of $\gamma_j$ implies that
  \[
    |\dot\gamma_j(t)|_{g(t,x'_j(t))}^2 = -\alpha(t,x'_j(t)) + |\dot x'_j(t)|_{\kappa(t,x'_j(t))}^2 \leq 0.
  \]
  Writing $w(t)=\delta e^{\digamma(t-\ft(q'))}$, we have
  \begin{align*}
    |\dot{\tilde\gamma}_{j,\delta}(t)|_{g(t+w(t),x'_j(t))}^2 &= -(1+w'(t))^2 \alpha(t+w(t),x'_j(t)) + |\dot x'_j(t)|_{\kappa(t+w(t),x'_j(t))}^2 \\
      &= |\dot\gamma_j(t)|_{g(t,x'_j(t))}^2 - \Bigl( (1+w'(t))^2 \alpha(t+w(t),x'_j(t)) - \alpha(t,x'_j(t)) \\
      &\hspace{12em} + |\dot x'_j(t)|_{\kappa(t+w(t),x'_j(t))}^2 - |\dot x'_j(t)|_{\kappa(t,x'_j(t))}^2\Bigr).
  \end{align*}
  We claim that the big parenthesis is nonnegative; indeed, denoting $\alpha_0=\inf_M\alpha>0$, and noting that $w'=\digamma w$, we can bound it from below for $t\in[t_j,\ft(q')]$ (so $0<w(t)\leq\delta$) by
  \[
    (2 w'(t)+w'(t)^2)\bigl(\alpha(t,x'_j(t))-C w(t)\bigr) - C w(t) \geq w\bigl( (2\digamma + \digamma^2 w)(\alpha_0-C\delta) - C\bigr) \geq \digamma\alpha_0 w > 0
  \]
  for some constant $C$ only depending on the metric $g$ (using that $\alpha$ and $\kappa$ are $\cC^1$), provided we fix $\digamma$ sufficiently large and consider $\delta\in(0,\delta_0]$ where $\delta_0>0$ is sufficiently small (and can be selected independently of $j$). The idea behind this construction is that one can dominate error terms of size $w$---arising from computing norms at points differing by $w(t)$ in the first coordinate---by a small constant times $w'$---arising in the expression for the tangent vector of $\tilde\gamma_{j,\delta}$.

  Furthermore, as $j\to\infty$ and $\delta>0$, we have $\tilde\gamma_{j,\delta}(\ft(q'))=\gamma_j(\ft(q'))+(\delta,0)\to q'$. Therefore, there exist $j_0\in\mathbb{N}$ and $\delta_0>0$ so that for all $j\geq j_0$ and $\delta\in(0,\delta_0]$, one can write $\tilde\gamma_{j,\delta}(\ft(q'))=\gamma_{q,\zeta_{j,\delta}}(1)$ where $\zeta_{j,\delta}\in T_q M$, $|\zeta_{j,\delta}|_{g^+}<2\eps$, is past timelike still. Fix such $\delta>0$. Then for sufficiently large $j$, the concatenation of $\tilde\gamma_{j,\delta}|_{[t_j,\ft(q')]}$ and $\gamma_{q,\zeta_{j,\delta}}(1-r)$, $r\in[0,1]$, and $\gamma|_{[s,\ft(q_0)]}$ is a piecewise smooth timelike curve connecting $(t_j+\delta,x_0)$ to $q_0$; this can be smoothed out to produce a \emph{smooth} timelike curve from $(t_j+\delta e^{-\digamma(\ft(q')-t_j)},x_0)$ to $q_0$. But as $j\to\infty$, we have $(t_j+\delta e^{-\digamma(\ft(q')-t_j)},x_0)\to(\bar t+\delta e^{-\digamma(\ft(q')-\bar t)},x_0)$; since $\bar t+\delta e^{-\digamma(\ft(q')-\bar t)}>\bar t$, we obtain a contradiction to the definition of $\bar t$.

  We conclude that $\zeta$ must be past \emph{lightlike}, and we re-define $\gamma$ by replacing the segment $\gamma|_{[s',s]}$ with the null-geodesic segment $\gamma_{q,\zeta}(1-r)$, $r\in[0,1]$, reparameterized by $\ft$. Repeat this construction for $s'$ in place of $s$: we obtain $s''\in(s_0,s')$ (so $s_0<s''<s'<s$), and a re-definition of $\gamma$ so that $\gamma''=\gamma|_{[s'',s']}$ and $\gamma'=\gamma|_{[s',s]}$ are null-geodesics. Consider the piecewise smooth curve $\gamma|_{[s'',s]}$; if it had a break point at $\gamma(s')$, i.e.\ if $\dot\gamma(s'+0)$ was not a scalar multiple of $\dot\gamma(s'-0)$, then by a standard short cut argument \cite[\S10]{ONeillSemi}, we could construct, for all sufficiently small $\delta>0$, a strictly timelike curve from any point in a $\delta$-neighborhood (with respect to $g^+$) of $\gamma(s'')$ to $\gamma(s)=q_0$, and then an argument as in the previous paragraph would produce a contradiction to the definition of $\bar t$. Therefore, $\gamma|_{[s'',s]}$ is, in fact, a smooth reparameterized null-geodesic.

  We can proceed in this fashion any finite number of times, starting at $s=\ft(q_0)$. (Note that we never reach a point on $\pa M$ in finitely many steps, as all constructions take place inside of $M^{\rm int}$.) Thus, we obtain $s_1\in[s_0,s')$ and a re-definition of $\gamma$ on $(s_1,\ft(q_0)]$ so that $\gamma|_{(s_1,\ft(q_0)]}$ is a smooth reparameterized null-geodesic. In the case that $s_1>s_0$, we can continue the construction at $\gamma(s_1)$. Therefore, we may in fact assume $s_1=s_0$. In conclusion, the thus re-defined curve $\gamma|_{[s_0,\ft(q_0)]}$ is a null-geodesic joining $\gamma(s_0)\in(0,T)\times\pa N$ with $\gamma(\ft(q_0))=q_0$; we define $\mu$ to be a smooth reparameterization of $\gamma|_{[s_0,\ft(q_0)]}$.

  Finally, we argue that $\mu$ has no cut points. If this were false, we could replace a segment of $\mu$ by another null-geodesic segment so that the resulting curve was not smooth. But then we could again appeal to a short cut argument and obtain a contradiction to the definition of $\bar t$. This finishes the proof.
\end{proof}

Denote now by $\gamma\colon[0,1]\to M$ a null-geodesic joining a point $\gamma(0)\in(0,T)\times\pa N$ to $q_0$. Consider $v_1:=-\dot{\gamma}(1)\in L^-_{q_0}M$, and let $\bft_1>0$ be such that $\gamma(0)=\gamma_{q_0,v_1}(\bft_1)$. For small $\delta>0$, we then have
\[
  z_1 := \gamma_{q_0,v_1}(\bft_1+\delta) \in (0,T)\times(N_1\setminus N);
\]
moreover, if we take $v_2,v_3,v_4\in L^-_{q_0}M$ to be very close (with respect to $g^+$) to $v_1$ and so that $v_1,v_2,v_3,v_4$ are linearly independent, then
\[
  z_j := \gamma_{q_0,v_j}(\bft_1+\delta) \in (0,T) \times (N_1\setminus N)
\]
as well. As initial velocities of null-geodesics at $z_j$, we then take
\[
  \zeta_j := -\dot\gamma_{q_0,v_j}(\bft_1+\delta).
\]
This construction guarantees that the $\gamma_{z_j,\zeta_j}$ intersect transversally at $q_0$, as desired; moreover, by taking $(z_j,\zeta_j)$ to be close enough to $(z_1,\zeta_1)$, they do not intersect prior to $q_0$, since otherwise $\gamma_{z_1,\zeta_1}$ would have to have a conjugate point on $[0,\bft_1+\delta]$, contradicting our assumption (for small enough $\delta$).

As a particular consequence of our arguments thus far, we can determine the collection
\[
  \{ \cE_M(q_0) \colon q_0 \in \mathbb{U}_g \cap M^{\rm int} \}
\]
of \emph{earliest} boundary light observation sets. An application of \cite[Theorem~1.2]{hintz2017reconstruction} then recovers $\mathbb{U}_g\cap M^{\rm int}$ as a smooth manifold, together with the conformal class of $g$. (In fact, the full strength of \cite[Theorem~1.2]{hintz2017reconstruction} is not needed, as we are only keeping track of the earliest light observation sets, which suffice for the reconstruction procedure of the reference to work.)

We would like to improve this and recover the structure of $\mathbb{U}_g$ as a smooth manifold with boundary $\mathbb{U}_g\cap\pa M=(0,T)\times\pa N$. In order to accomplish this, we take advantage of the extended manifold $M_1\supset(0,T)\times\pa N$ as follows. First of all, we can extend $\cE_M(q_0)$ for $q_0\in\mathbb U_g\cap M^{\rm int}$ to
\[
  \cE_{M_1}(q_0) \subset M_1 \setminus M^{\rm int}
\]
by taking the union of $\cE_M(q_0)$ with $\gamma_{z,\zeta}([0,\delta))$ where $z\in\cE_M(q_0)$, $\zeta$ is the unique (up to a positive scalar multiple) outward pointing null vector normal to $\cE_M(q_0)$ at $z$ (see \cite[Lemma~2.5]{hintz2017reconstruction}), and $\delta>0$ is arbitrary unless $\gamma_{z,\zeta}$ reenters $(0,T)\times N$, in which case we take $\delta=\bft_1(z,\zeta)$. Thus, the intersection with the exterior of $M$, i.e.\ $\cE_{M_1}(q_0)\cap M_1\setminus M$, is a smooth null hypersurface in a sufficiently small neighborhood of any $z\in\cE_M(q_0)$.

For $q_0\in(0,T)\times\pa N$, we can similarly consider all tangent vectors $\zeta\in L^+_{q_0}M$ which point out of $M$, and take the union $\cE_{M_1}(q_0)$ over all null-geodesic segments $\gamma_{q_0,\zeta}([0,\delta))$ with $\delta>0$ defined in the same way as before.

In order to recover the smooth structure of $\mathbb{U}_g$ near $q_0\in(0,T)\times\pa N$, we now follow the strategy of \cite{hintz2017reconstruction}. Namely, the earliest observation time along $4$ generic timelike curves in $M_1\setminus M$ defined in a neighborhood of $q_0$ provide smooth local coordinates in the intersection of a small neighborhood $\cU\subset M$ of $q_0$ with $M^{\rm int}$; and they remain smooth down to $\cU\cap\pa M$. This produces a smooth local coordinate system on $\cU$, as desired. This completes the reconstruction of $\mathbb{U}_g$ as a smooth manifold and the conformal class of $g$ on it.

\begin{remark}
  An alternative procedure constructs the intersection with $M_1\setminus M$ of the future light cones based at all points
  \begin{equation}
  \label{EqAllPoints}
    q_0\in\mathbb U_g\cup((0,T)\times(N_1\setminus N));
  \end{equation}
  the idea is that for a null-geodesic starting at $q_0$ which strikes $(0,T)\times\pa N$, we can determine (from the Dirichlet-to-Neumann map) the place and tangent vector on $(0,T)\times\pa N$ where it exits (if it exits at all), from where one can extend it using the metric $g$ in $M_1\setminus M$. Thus, we can construct the (earliest) light observation sets of all points~\eqref{EqAllPoints} as measured in $M_1\setminus M$; a straightforward application of the methods of \cite{kurylev2018inverse} then reconstructs the set in~\eqref{EqAllPoints} as a smooth manifold equipped with the conformal class of a Lorentzian metric.
\end{remark}

\subsection{Determination of the conformal factor}
\label{SsFact}

From now on, we just assume, without loss of generality, $g=e^{2\beta}\widetilde{g}$ with $\beta=0$ on $M_1\setminus M$; and we restrict to $\mathbb{U}_g=\mathbb{U}_{\widetilde{g}}$. Let $\Sigma_g$ be the characteristic set, i.e.,
\[
\Sigma_g=\{(x,\xi)\in T^*M_1 \colon |\xi|^2_{g(x)^{-1}}=0\}.
\]
Note that $\Sigma_g=\Sigma_{\widetilde{g}}$.
 Denote the Lagrangians $\Lambda^g$ and $\Lambda^{\widetilde{g}}$ to be the flowouts of $\Sigma_g$ under the Hamiltonians associated with $g$ and $\widetilde{g}$. By \cite[Proposition 4.5]{lassas2018inverse} we have $\Lambda^g=\Lambda^{\widetilde{g}}$ and the principal symbols of $\mathcal{Q}_g,\mathcal{Q}_{\widetilde{g}}\in \mathcal{I}^{-2}(N^*\mathrm{diag}\setminus\Lambda^g)$ satisfy
\[
\sigma_p(\mathcal{Q}_g)=e^{2\beta}\sigma_p(\mathcal{Q}_{\widetilde{g}}).
\]
The principal symbols in $\mathcal{I}^{-3/2}(\Lambda^g\setminus N^*\mathrm{diag})$ satisfy
\[
\sigma_p(\mathcal{Q}_g)(x,\xi,y,\eta)=e^{-\beta(x)}\sigma_p(\mathcal{Q}_{\tilde{g}})(x,\xi,y,\eta)e^{3\beta(y)},
\]
where $(y,\eta)$ is joined to $(x,\xi)$ by a null-bicharacteristic on $\Lambda^g$. Therefore, the solutions $v_i$, resp.\ $\wt v_i$ of the linear initial-boundary value problems~\eqref{Eqvj} with respect to $g$, resp.\ $\wt g$ with the same boundary source $f_j$ satisfy
\[
\sigma_p(v_i)(q_0,\theta_j)=e^{-\beta(q_0)}\sigma_p(\widetilde{v}_i)(q_0,\theta_j),
\]
for $\theta_j\in N_{q_0}^*K_j$.
By \cite[Proposition 3.12]{lassas2018inverse}, we have (cf. \eqref{symbol_Uinc})
\[
\sigma_p(\widetilde{\mathcal{U}}^{\mathrm{inc}})(y,\eta)=-\frac{1}{24}(2\pi)^{-3}\sigma_p(\mathcal{Q}_{\widetilde{g}})(y,\eta,q_0,\zeta^\flat)\wt a(q_0)\prod_{i=1}^4\sigma_p(\widetilde{v}_i)(q_0,\theta_j)\\
\]
and
\[
\begin{split}
\sigma_p(\mathcal{U}^{\mathrm{inc}})(y,\eta)&=-\frac{1}{24}(2\pi)^{-3}\sigma_p(\mathcal{Q}_g)(y,\eta,q_0,\zeta^\flat)a(q_0)\prod_{i=1}^4\sigma_p(v_i)(q_0,\theta_j)\\
&=-\frac{1}{24}(2\pi)^{-3}\sigma_p(\mathcal{Q}_{\tilde{g}})(y,\eta,q_0,\zeta^\flat)e^{3\beta(q_0)}a(q_0)e^{-4\beta(q_0)}\prod_{i=1}^4\sigma_p(\widetilde{v}_i)(q_0,\theta_j)\\
&=-\frac{1}{24}(2\pi)^{-3}\sigma_p(\mathcal{Q}_{\tilde{g}})(y,\eta,q_0,\zeta^\flat)a(q_0)e^{-\beta(q_0)}\prod_{i=1}^4\sigma_p(\widetilde{v}_i)(q_0,\theta_j).
\end{split}
\]
Therefore $\sigma_p(\widetilde{\mathcal{U}}^{\mathrm{inc}})(y,\eta)=\sigma_p(\mathcal{U}^{\mathrm{inc}})(y,\eta)$ implies that $\wt a(q_0)=a(q_0)e^{-\beta(q_0)}$. If we assume a priori that $a(q_0)=\wt a(q_0)=c\neq 0$ for all $q_0$, then this implies $e^{-\beta(q_0)}=1$, hence $\beta(q_0)=0$ for all $q_0$.

This finishes the proof of a large part of Theorem~\ref{ThmI}; it only remains to prove that $\Box_g e^\beta=0$, which we will do in the next section.

\section{Reconstruction using Gaussian beams}\label{gaussianbeam}

In this section, we will prove:
\begin{prop}\label{conf_uniqueness}
Suppose we are given two smooth Lorentzian metrics $g,\widetilde{g}$ and two smooth functions $a,\widetilde{a}$ on $M$ of the form~\eqref{EqIMetric}, and assume that null-geodesics in $(\mathbb U_g,g)$, resp.\ $(\mathbb U_{\wt g},\wt g)$ have no conjugate points. Assume further that $\widetilde{g}$ and $g$ are in the same conformal class, i.e., there exists a smooth function $\beta$ on $M$ such that $\widetilde{g}=e^{-2\beta}g$. 
If $\Lambda_{\widetilde{g},\widetilde{a}}=\Lambda_{g,a}$, then we have $\widetilde{a}=e^{-\beta}a$ and $\Box_g e^{-\beta}=0$ in $\mathbb{U}_g$.
\end{prop}

By the boundary determination, we have $\beta\vert_{(0,T)\times\partial N}=\partial_\nu\beta\vert_{(0,T)\times\partial N}=0$. Denote by $\Lambda_{g,q,a}$ the Dirichlet-to-Neumann map for the equation $\square_gu+qu+au^4=0$. We start from the fact, shown in the proof of Lemma~\ref{conf_invar}, that
\[
\Lambda_{\widetilde{g},\widetilde{a}}=\Lambda_{g,q,e^\beta\widetilde{a}}=\Lambda_{g,0,a}=\Lambda_{g,a},
\]
where $-q=e^{\beta}\square_ge^{-\beta}$. (See also \cite[Lemma~2]{stefanov2018inverse}.)
Therefore, we only need to consider the equations
\[
\square_g\widetilde{v}+q\widetilde{v}+e^\beta\widetilde{a}\widetilde{v}^4=0,
\]
and
\[
\square_gv+av^4=0,
\]
on the same Lorentzian manifold $(M,g)$. It is easy to see that Proposition \ref{conf_uniqueness} is then a direct consequence of the following lemma.
\begin{lemma}\label{lemma_q}
Assume that null-geodesics in $(\mathbb U_g,g)$ have no conjugate points. If $\Lambda_{g,q,e^\beta\widetilde{a}}=\Lambda_{g,0,a}$, then $q=0,\, \widetilde{a}=e^{-\beta} a$ in $\mathbb{U}_g$.
\end{lemma}

Instead of using distorted plane waves as in Section \ref{interiordetermination}, we will use Gaussian beam solutions to prove Proposition \ref{conf_uniqueness}. The Gaussian beams have also been used for various inverse problems for both elliptic an hyperbolic equations \cite{katchalov1998multidimensional,bao2014sensitivity, belishev1996boundary, dos2016calderon,feizmohammadi2019timedependent,feizmohammadi2019recovery,feizmohammadi2019inverse}. Note that the fact $\widetilde{a}=e^{-\beta}a$ has already been proved in previous section using distorted plane waves, and this section includes an alternative proof.
\subsection{Construction of Gaussian beam solutions}
We will construct Gaussian beam solutions \cite{RalstonLocalized} for the \textit{linear} equation 
\[
\square_g\widetilde{u}+q\widetilde{u}=0, \quad (\text{and}\quad \square_gu=0).
\]
A Gaussian beam can be thought as a wave packet traveling along a null-geodesic $\gamma$. We first construct Gaussian beam solutions on the manifold $M_1$ without boundary.
The construction can be done in Fermi coordinates in a neighborhood of $\gamma$. Assume that $\gamma$ passes through a point $p\in M$ and joins two points $\gamma(\tau_-)$ and $\gamma(\tau_+)$ on the boundary $\mathbb{R}\times\partial N$. We will use Fermi coordinates $\Phi$ on the extended manifold $M_1$ in a neighborhood of $\gamma([\tau_-,\tau_+])$, denoted by $(z^0:=\tau,z^1,z^2,z^3)$, such that $\Phi(\gamma(\tau))=(\tau,0)$ (cf.\ \cite[Lemma 1]{feizmohammadi2019recovery}). 

The Gaussian beams are of the form
\[
\widetilde{u}_\rho(x)=e^{\mathrm{i}\rho\varphi(x)}\widetilde{\mathfrak{a}}_\rho(x),
\]
with
\begin{equation}\label{phi_a}
\varphi=\sum_{k=0}^N\varphi_k(\tau,z'),\quad\widetilde{\mathfrak{a}}_\rho(\tau,z')=\chi\left(\frac{|z'|}{\delta}\right)\sum_{k=0}^N\rho^{-k}\widetilde{a}_k(\tau,z'),\quad
\widetilde{a}_k(\tau,z')=\sum_{j=0}^N\widetilde{a}_{k,j}(\tau,z')
\end{equation}
in a neighborhood of $\gamma$,
\begin{equation}\label{neighbor_V}
\mathcal{V}=\bigl\{(\tau,z')\in M_1 : \tau\in \bigl[\tau_--\tfrac{\epsilon}{\sqrt{2}},\tau_++\tfrac{\epsilon}{\sqrt{2}}\bigr],\,|z'|<\delta\bigr\}.
\end{equation}
Here for each $j$, $\varphi_j$ and $\widetilde{a}_{k,j}$ are a complex valued homogeneous polynomials of degree $j$ with respect to the variables $z^i$, $i=1,2,3$, and $\delta>0$ is a small parameter. The smooth function $\chi:\mathbb{R}\rightarrow [0,+\infty)$ satisfies $\chi(t)=1$ for $|t|\leq\frac{1}{4}$ and $\chi(t)=0$ for $|t|\geq \frac{1}{2}$. 

By calculation, one can verify that
\begin{equation}
\begin{split}
(\square_g+q)(\widetilde{\mathfrak{a}}_\rho e^{\mathrm{i}\rho\varphi})&=e^{\mathrm{i}\rho\varphi}(\rho^2(\mathcal{S}\varphi)\widetilde{\mathfrak{a}}_\rho-\mathrm{i}\rho\mathcal{T}\widetilde{\mathfrak{a}}_\rho+(\square_g+q)\widetilde{\mathfrak{a}}_\rho), \\
  &\quad \mathcal{S}\varphi=\langle \mathrm{d}\varphi,\mathrm{d}\varphi\rangle_g, \\
  &\quad \mathcal{T}\widetilde{a}=2\langle \mathrm{d}\varphi,\mathrm{d}\widetilde{a}\rangle_g-(\square_g\varphi)\widetilde{a}.
\end{split}
\end{equation}
We construct the phase $\varphi$ and the amplitude $\widetilde{\mathfrak{a}}_\rho$ such that
\begin{equation}\label{S_cond1}
\frac{\partial^\Theta}{\partial z^\Theta}(\mathcal{S}\varphi)(\tau,0)=0,
\end{equation}
\begin{equation}\label{S_cond2}
\frac{\partial^\Theta}{\partial z^\Theta}(\mathcal{T}\widetilde{a}_0)(\tau,0)=0, \quad
\end{equation}
\begin{equation}\label{S_cond3}
\frac{\partial^\Theta}{\partial z^\Theta}(-\mathrm{i}\mathcal{T}\widetilde{a}_k+(\square_g+q) \widetilde{a}_{k-1})(\tau,0)=0
\end{equation}
for $\Theta=(\Theta_0=0,\Theta_1,\Theta_2,\Theta_3)$ with $|\Theta|\leq N$. For more details we refer to \cite{feizmohammadi2019recovery}. For the construction of the phase function $\varphi$, we can take
\[
\varphi_0=0,\quad \varphi_1=z^1,\quad
\varphi_2(\tau,z)=\sum_{1\leq i,j\leq 3}H_{ij}(\tau)z^iz^j.
\]
Here $H$ is a symmetric matrix with $\Im H(\tau)>0$; the matrix $H$ satisfies a Riccati ODE,
\begin{equation}\label{Riccati}
\frac{\mathrm{d}}{\mathrm{d}\tau}H+HCH+D=0,\quad\tau\in \bigl(\tau_--\tfrac{\epsilon}{2},\tau_++\tfrac{\epsilon}{2}\bigr),\quad H(0)=H_0,\text{ with }\Im H_0>0,
\end{equation}
where $C$, $D$ are matrices with $C_{11}=0$, $C_{ii}=2$, $i=2,3$, $C_{ij}=0$, $i\neq j$ and $D_{ij}=\frac{1}{4}(\partial_{ij}^2g^{11})$.
 
 \begin{lemma}[\text{\cite[Lemma 3.2]{feizmohammadi2019timedependent}}]\label{lemma_R0}
 The Riccati equation \eqref{Riccati} has a unique solution. Moreover the solution $H$ is symmetric and $\Im (H(\tau))>0$ for all $\tau\in (\tau_--\frac{\delta}{2},\tau_++\frac{\delta}{2})$. For solving the above Riccati equation, one has $H(\tau)=Z(\tau)Y(\tau)^{-1}$, where $Y(\tau)$ and $Z(\tau)$ solve the ODEs
 \begin{alignat*}{2}
 \frac{\mathrm{d}}{\mathrm{d}\tau}Y(\tau)&=CZ(\tau),&\quad Y(\tau_-)&=Y_0,\\
\frac{\mathrm{d}}{\mathrm{d}\tau}Z(\tau)&=-D(\tau)Y(\tau),&\quad Z(\tau_-)&=Y_1=H_0 Y_0.
 \end{alignat*}
 In addition, $Y(\tau)$ is nondegenerate.
 \end{lemma}

 \begin{lemma}[\text{\cite[Lemma 3.3]{feizmohammadi2019timedependent}}]\label{lemma_H0}
 The following identity holds:
\[
 \det(\Im(H(\tau))|\det(Y(\tau))|^2=c_0
\]
 with $c_0$ independent of $\tau$.
 \end{lemma}
  We see that the matrix $Y(\tau)$ satisfies
  \begin{equation}\label{eq_Y}
  \frac{\mathrm{d}^2}{\mathrm{d}\tau^2}Y+CD Y=0,\quad Y(\tau_-)=Y_0,\quad \frac{\mathrm{d}}{\mathrm{d}\tau}Y(\tau_-)=CY_1.
  \end{equation}

As in \cite{feizmohammadi2019recovery}, we have the following estimate by the construction of $\widetilde{u}_\rho$
\begin{equation}\label{est_remainder}
(\|\square_g+q) \widetilde{u}_\rho\|_{H^k(M)}\leq C\rho^{-K},\qquad K=\frac{N+1-k}{2}-1.
\end{equation}

For the amplitude, we first notice that the principal term $\widetilde{a}_0$, which satisfies \eqref{S_cond2}, is independent of $q$. Setting $|\Theta|=0$ in \eqref{S_cond2}, one can get the equation for $\widetilde{a}_{0,0}$
\[
2\frac{\mathrm{d}}{\mathrm{d}\tau}\widetilde{a}_{0,0}+\mathrm{Tr}(CH)\widetilde{a}_{0,0}=0.
\]
By Lemma \ref{lemma_R0} and Lemma \ref{lemma_H0}, we have
\[
\mathrm{Tr}(CH)=\frac{\mathrm{d}}{\mathrm{d}\tau}\log\det Y.
\]
Therefore, we can take
\[
\widetilde{a}_{0,0}(\tau,0)=\det(Y(\tau))^{-\frac{1}{2}}.
\]
For more details, we refer to \cite{feizmohammadi2019recovery}. The terms $\widetilde{a}_{0,k}$, $j=1,2,\cdots, N$ can be constructed successively by solving linear first order ODEs resulted from $\eqref{S_cond2}$. This finishes the construction of $\widetilde{a}_0$.

For the construction of $\widetilde{a}_1$, we use equation \eqref{S_cond3}. Set $k=1$ and $|\Theta|=0$, then
\begin{equation}\label{eq_a1}
2\frac{\partial}{\partial\tau}\widetilde{a}_{1,0}(\tau,0)+\mathrm{Tr}(CH)\widetilde{a}_{1,0}(\tau,0)=-\mathrm{i}(\square_g+q)\widetilde{a}_0(\tau,0).
\end{equation}
Therefore we can take
\[
\widetilde{a}_{1,0}(\tau,0)=-\frac{\mathrm{i}}{2}\det(Y(\tau))^{-\frac{1}{2}}\int_{\tau_-}^\tau\left[(\square_g\widetilde{a}_0)(\tau',0)+q(\tau',0))\det(Y(\tau')^{\frac{1}{2}}\right]\mathrm{d}\tau'.
\]
Notice that $\widetilde{a}_{1,0}(\tau_-,0)=0$, and $\frac{\partial}{\partial\tau}\widetilde{a}_{1,0}(\tau_-,0)$ depends on $q$ only via $q(\tau_-,0)$. Inductively one can prove that $\frac{\partial^j}{\partial\tau^j}\widetilde{a}_{1,0}(\tau_-,0)$ depends on $q$ only via $\frac{\partial^k}{\partial\tau^k}q(\tau_-,0)$ for $k=0,1,\ldots,j$ for each $j\in \mathbb{N}$ by \eqref{eq_a1}. For any $|\Theta|=k$, $1\leq k\leq N$, we have
\begin{equation}\label{eq_a1k}
2\frac{\partial}{\partial\tau}\frac{\partial^\Theta}{\partial z^\Theta}\widetilde{a}_{1,k}(\tau,0)+\mathrm{Tr}(CH)\frac{\partial^\Theta}{\partial z^\Theta}\widetilde{a}_{1,k}(\tau,0)=\mathcal{E}_k,
\end{equation}
where $\mathcal{E}_k$ depends on $g,q,\varphi,\widetilde{a}_0$ along with their derivatives on $\gamma$, and $\frac{\partial^\alpha}{\partial z^\alpha}\widetilde{a}_{1}(\tau,0),|\alpha|\leq k-1$. We can take 
\[
\frac{\partial^\Theta}{\partial z^\Theta}\widetilde{a}_{1,k}(\tau_-,0)=0,
\]
and solve the equation \eqref{eq_a1k} for $\frac{\partial^\Theta}{\partial z^\Theta}\widetilde{a}_{1,k}(\tau_,0)$. This determines $\widetilde{a}_1$.
By a similar argument as above, we conclude that $\frac{\partial^j}{\partial\tau^j}\frac{\partial^\Theta}{\partial z^\Theta}\widetilde{a}_{1}(\tau_-,0)$ depends on $q$ only via the jets of $q$ at $(\tau_-,0)$, for any $j\in\mathbb{N}$. The functions $\widetilde{a}_2,\ldots, \widetilde{a}_N$ can be determined in a similar way.
This completes the construction of the Gaussian beam solution $\widetilde{u}_\rho$.

\noindent\textbf{Construction of Gaussian beams with reflections at the boundary. } Next, we take into account the reflections of Gaussian beams at the boundary $\partial M$.
Now assume $\gamma:[0,L]$ is a broken null-geodesic with $0<\mathbf{t}_1<\mathbf{t}_2<\cdots<\mathbf{t}_N<L$ the times of reflections at the boundary. We will construct Gaussian beam solutions concentrating near $\gamma$. Gaussian beams with reflections on Riemannian manifolds are constructed in \cite{kenig2014calderon}. We only give details about how to deal with the first reflection at time $\mathbf{t}_1$, and all the subsequent reflections can be dealt in the same way. Consider Gaussian beam solutions of the form
\[
\widetilde{u}_\rho=\widetilde{u}_\rho^{\mathrm{inc}}+\widetilde{u}_\rho^{\mathrm{ref}},
\]
where $\widetilde{u}_\rho^{\mathrm{inc}}$ is the Gaussian beam solution associated with the geodesic $\gamma\vert_{[0,\mathbf{t}_1]}$, and $\widetilde{u}_\rho^{\mathrm{ref}}$ is the Gaussian beam solution associated with the geodesic $\gamma\vert_{[\mathbf{t}_1,\mathbf{t}_2]}$. We will construct $\widetilde{u}_\rho^{\mathrm{inc}}$ and $\widetilde{u}_\rho^{\mathrm{ref}}$ such that $(\widetilde{u}_\rho^{\mathrm{inc}}+\widetilde{u}_\rho^{\mathrm{ref}})\vert_{\partial M}$ is small near $\gamma(\mathbf{t}_1)$.
Here we denote
\[
\widetilde{u}_\rho^{\mathrm{inc}}=e^{\mathrm{i}\rho\varphi^{\mathrm{inc}}}\widetilde{\mathfrak{a}}^{\mathrm{inc}}_\rho,\quad\quad \widetilde{u}_\rho^{\mathrm{ref}}=e^{\mathrm{i}\rho\varphi^{\mathrm{ref}}}\widetilde{\mathfrak{a}}^{\mathrm{ref}}_\rho,
\]
both of the form \eqref{phi_a}.
 We can take $\varphi^{\mathrm{ref}}\vert_{\partial M}=\varphi^{\mathrm{inc}}\vert_{\partial M}$ up to $N$-th order at $\gamma(\mathbf{t}_1)\in\partial M$. In particular, we have
\[
\begin{split}
\mathrm{d}(\varphi^{\mathrm{inc}}\vert_{\partial M})\vert_{\gamma(\mathbf{t}_1)}=\dot{\gamma}(\mathbf{t}_1-)^\flat\vert_{T\partial M},\\
\mathrm{d}(\varphi^{\mathrm{ref}}\vert_{\partial M})\vert_{\gamma(\mathbf{t}_1)}=\dot{\gamma}(\mathbf{t}_1+)^\flat\vert_{T\partial M}.
\end{split}
\]
We can also choose $\widetilde{\mathfrak{a}}^{\mathrm{ref}}\vert_{\partial M}=-\widetilde{\mathfrak{a}}^{\mathrm{inc}}\vert_{\partial M}$ up to $N$-th order at $\gamma(\mathbf{t}_1)$.

Let $R_1$ be a small neighborhood of $\gamma(\mathbf{t}_1)$ on $\partial M$ such that $\widetilde{u}_\rho^{\mathrm{inc}}$ and $\widetilde{u}_\rho^{\mathrm{ref}}$ are compactly supported in $R_1$. Let $(y,x_n)$ be the coordinates near $\gamma(\mathbf{t}_1)$ such that $\partial M$ is parametrized by $y\mapsto (y,x_n)$. First notice that on $R_1$,
\[
|e^{\mathrm{i}\rho\varphi^\bullet}|\leq e^{-C'\rho|y|^2}.
\]
By above considerations, we have
\[
|\varphi^{\mathrm{ref}}-\varphi^{\mathrm{inc}}|\leq C|y|^{N+1},
\]
and
\[
|\widetilde{\mathfrak{a}}_\rho^{\mathrm{ref}}+\widetilde{\mathfrak{a}}^{\mathrm{inc}}_\rho|\leq C|y|^{N+1}
\]
on $R_1$.
We can write $\widetilde{u}_\rho$ as
\[
\widetilde{u}_\rho=(e^{\mathrm{i}\rho\varphi^{\mathrm{inc}}}-e^{\mathrm{i}\rho\varphi^{\mathrm{ref}}})\widetilde{\mathfrak{a}}_\rho^{\mathrm{inc}}+e^{\mathrm{i}\rho\varphi^{\mathrm{ref}}}(\widetilde{\mathfrak{a}}_\rho^{\mathrm{ref}}+\widetilde{\mathfrak{a}}_\rho^{\mathrm{inc}}).
\]
We have
\[
e^{\mathrm{i}\rho\varphi^{\mathrm{inc}}}-e^{\mathrm{i}\rho\varphi^{\mathrm{ref}}}=\mathrm{i}\rho(\varphi^{\mathrm{inc}}-\varphi^{\mathrm{ref}})\int_0^1e^{\mathrm{i}\rho(s\varphi^{\mathrm{inc}}+(1-s)\varphi^{\mathrm{ref}})}\mathrm{d}s
\]
and consequently
\[
|e^{\mathrm{i}\rho\varphi^{\mathrm{inc}}}-e^{\mathrm{i}\rho\varphi^{\mathrm{ref}}}|\leq C\rho |y|^{N+1}e^{-C'\rho|y|^2}.
\]
Thus we obtain that
\[
\vert\widetilde{u}_\rho\vert_{R_1}\vert\leq C\rho |y|^{N+1}e^{-C'\rho|y|^2}.
\]

Taking the first derivative of $\widetilde{u}_\rho$ in $y$, we have
\[
D_y\widetilde{u}_\rho=\mathrm{i}\rho e^{\mathrm{i}\rho\varphi^{\mathrm{inc}}}D_y\varphi^{\mathrm{inc}}\widetilde{\mathfrak{a}}^{\mathrm{inc}}_\rho+e^{\mathrm{i}\rho\varphi^{\mathrm{inc}}}D_y\widetilde{\mathfrak{a}}^{\mathrm{inc}}_\rho+\mathrm{i}\rho e^{\mathrm{i}\rho\varphi^{\mathrm{ref}}}D_y\varphi^{\mathrm{ref}}\widetilde{\mathfrak{a}}^{\mathrm{ref}}_\rho+e^{\mathrm{i}\rho\varphi^{\mathrm{ref}}}D_y\widetilde{\mathfrak{a}}^{\mathrm{ref}}_\rho.
\]
Notice that $|D_y\varphi^{\mathrm{inc}}|\leq C|y|$ on $R_1$, we obtain the estimate
\[
\vert D_y\widetilde{u}_\rho\vert_{R_1}\vert\leq C\rho  |y|^{N+2}e^{-C'\rho|y|^2}+C|y|^{N}e^{-C'\rho|y|^2}.
\]
Inductively, one can show that 
\[
\left\vert\partial^\alpha_y\widetilde{u}_\rho\vert_{R_1}\right\vert\leq C\sum_{k+j=|\alpha|}\rho^{k}|y|^ke^{-C'\rho|y|^2}|y|^{N+1-j}.
\]
By change of variables $y=\rho^{-1/2}y'$, we obtain
\[
\int_{R_1}|\partial^\alpha_y\widetilde{u}_\rho\vert^2\mathrm{d}S\mathrm{d}t\leq C \rho^{-(N-|\alpha|+1)-\frac{3}{2}},
\]
and therefore
\begin{equation}\label{u_tilde_R1}
\|\widetilde{u}_\rho\|_{H^k(R_1)}\leq C\rho^{-\frac{N-k+1}{2}-\frac{3}{4}}.
\end{equation}

We can also construct the Gaussian beam solution associated with the same broken null-geodesic $\gamma$ for the equation $\square_gu=0$ by just setting $q=0$ in the above process. Considering only the reflection at $\gamma(\mathbf{t}_1)$, we denote the solution to be
\[
u_\rho=u_\rho^{\mathrm{inc}}+u_\rho^{\mathrm{ref}}.
\]

Since the incidence waves are what will be used explicitly, we shall denote
\[
\widetilde{u}_\rho^{\mathrm{inc}}(x)=e^{\mathrm{i}\rho\varphi^{\mathrm{inc}}(x)}\widetilde{\mathfrak{a}}^{\mathrm{inc}}_\rho(x),\quad u_\rho^{\mathrm{inc}}(x)=e^{\mathrm{i}\rho\varphi^{\mathrm{inc}}(x)}\mathfrak{a}^{\mathrm{inc}}_\rho(x),
\]
with
\[
\widetilde{\mathfrak{a}}^{\mathrm{inc}}_\rho(\tau,z')=\chi\left(\frac{|z'|}{\delta}\right)\sum_{k=0}^N\rho^{-k}\widetilde{a}_k(\tau,z'),\quad
\widetilde{a}_k(\tau,z')=\sum_{j=0}^N\widetilde{a}_{k,j}(\tau,z').
\]
\[
\mathfrak{a}^{\mathrm{inc}}_\rho(\tau,z')=\chi\left(\frac{|z'|}{\delta}\right)\sum_{k=0}^N\rho^{-k}a_k(\tau,z'),\quad
a_k(\tau,z')=\sum_{j=0}^Na_{k,j}(\tau,z').
\]
In above notations we have used the fact that the phase function $\varphi^{\mathrm{inc}}$, since independent of $q$, is the same for $\widetilde{u}^{\mathrm{inc}}_\rho$ and $u^{\mathrm{inc}}_\rho$. 
We remark here that the following estimate, analogous to \eqref{u_tilde_R1}, also holds for $u_\rho$,
\begin{equation}\label{u_R1}
\|u_\rho\|_{H^k(R_1)}\leq C\rho^{-\frac{N-k+1}{2}-\frac{3}{4}}.
\end{equation}
Also we have $\widetilde{a}_0=a_0$, as they are also independent of $q$.
Now we have
\[
a_{1}(\tau,0)=a_{1,0}(\tau,0)=-\frac{\mathrm{i}}{2}\det(Y(\tau))^{-\frac{1}{2}}\int_{\mathbf{t}_0}^\tau\left[(\square_ga_0)(\tau',0)\det(Y(\tau'))^{\frac{1}{2}}\right]\mathrm{d}\tau'.
\]

Recall that the jet of $q$ vanishes at $\gamma(\mathbf{t}_0)\in\partial M$. By above discussion, $(\widetilde{\mathfrak{a}}_\rho-\mathfrak{a}_\rho)\vert_{\partial M}$ vanishes at $\gamma(\mathbf{t}_0)$ up to $N$-th order. Then we have
\[
\left\vert\partial^\alpha_y(\widetilde{u}_\rho-u_\rho)\right\vert\leq C\sum_{k+j=|\alpha|}\rho^{k}|y|^ke^{-C'\rho|y|^2}|y|^{N+1-j}.
\]
Similar as \eqref{u_tilde_R1} and \eqref{u_R1}, we have
\begin{equation}\label{u_difference_R0}
\|\widetilde{u}_\rho-u_\rho\|_{H^k(R_0)}\leq C\rho^{-\frac{N-k+1}{2}-\frac{3}{4}},
\end{equation}
where $R_0$ is a small neighborhood of $\gamma(\mathbf{t}_0)$ on $\partial M$, and $\widetilde{u}_\rho\vert_{\partial M},u_\rho\vert_{\partial M}$ are supported in $R_0$.
Combining the estimates \eqref{u_tilde_R1}, \eqref{u_R1}, and \eqref{u_difference_R0}, we finally get
\[
\|\widetilde{u}_\rho-u_\rho\|_{H^k(\partial M)}\leq C\rho^{-\frac{N-k+1}{2}-\frac{3}{4}}.
\]
In conclusion, for any $K\geq 0$, we can take $N$ large enough such that
\[
\|\widetilde{u}_\rho-u_\rho\|_{H^k(\partial M)}\leq C\rho^{-K},
\]
which is valid with an arbitrary number of reflections of $\gamma$.
\subsection{Proof of Lemma \ref{lemma_q}}
For any point $x\in\partial M$, and an inward pointing null-vector $\xi\in L^{+}_xM_1$, consider the broken null-geodesic $\gamma=\gamma_{x,\xi}$ starting from $(x,\xi)$. By the discussions above, we can construct boundary sources $f_{\rho,x,\xi}$ such that the solutions $\widetilde{u}$, $u$ to the equations
\begin{equation}\label{eq_tildeu}
\begin{alignedat}{2}
  (\Box_g +q)\widetilde{u} &= 0, &\quad & \text{on }M,\\
  \widetilde{u} &= f_{\rho,x,\xi}(x), &\quad& \text{on }\partial M,\\
  \widetilde{u}(\ft,x') &= 0, &\quad& \ft<0,
\end{alignedat}
\end{equation}
and
\begin{equation}\label{eq_u}
\begin{alignedat}{2}
  \Box_gu &= 0, &\quad & \text{on }M,\\
  u &= f_{\rho,x,\xi}(x), &\quad& \text{on }\partial M,\\
 u(\ft,x') &= 0, &\quad& \ft<0,
\end{alignedat}
\end{equation}
satisfy
\[
\widetilde{u}=\widetilde{u}_\rho+\widetilde{R}_\rho,\quad\quad\quad\quad u=u_\rho+R_\rho,
\]
where $\widetilde{u}_\rho,u_\rho$ are the Gaussian beam solutions associated with the broken null-geodesic $\gamma_{x,\xi}$ constructed above, and
\begin{equation}\label{remainder1}
\|f_{\rho,x,\xi}(x)-\widetilde{u}_\rho\|_{H^k(\partial M)}=\mathcal{O}(\rho^{-K}),\quad\|f_{\rho,x,\xi}(x)-u_\rho\|_{H^k(\partial M)}=\mathcal{O}(\rho^{-K}).
\end{equation}
For example, one can take $f_{\rho,x,\xi}=\widetilde{u}_\rho\vert_{\partial M}$ or $f_{\rho,x,\xi}=u_\rho\vert_{\partial M}$, and let the remainder terms be the solutions to the equations
\[
\begin{alignedat}{2}
  (\Box_g +q)\widetilde{R}_\rho &= -(\Box_g +q)\widetilde{u}_\rho, &\quad & \text{on }M,\\
  \widetilde{R}_\rho &= f_{\rho,x,\xi}(x)-\widetilde{u}_\rho, &\quad& \text{on }\partial M,\\
  \widetilde{R}_\rho(\ft,x') &= 0, &\quad& \ft<0,
\end{alignedat}
\]
and
\[
\begin{alignedat}{2}
  \Box_gR_\rho &= -(\Box_g +q)u_\rho, &\quad & \text{on }M,\\
  R_\rho &= f_{\rho,x,\xi}(x)-u_\rho(x), &\quad& \text{on }\partial M,\\
 R_\rho(\ft,x') &= 0, &\quad& \ft<0,
\end{alignedat}
\]
The parameter $\delta$ in \eqref{phi_a} can be taken small enough to ensure that $\widetilde{u}_\rho=0$, $u_\rho=0$ near $\{t=0\}$.
By \eqref{est_remainder}, we have
\begin{equation}\label{remainder2}
\|(\Box_g +q)\widetilde{u}_\rho\|_{H^k(M)}=\mathcal{O}(\rho^{-K}),\quad\|\Box_g u_\rho\|_{H^k(M)}=\mathcal{O}(\rho^{-K}).
\end{equation}
By \eqref{remainder1} and \eqref{remainder2}, we have
\[
\|\widetilde{R}_\rho\|_{H^{k+1}(M)},\ \|R_\rho\|_{H^{k+1}(M)}\leq C\rho^{-K}.
\]

\medskip

Fix a point $q_0\in\mathbb{U}_g$. There exist $\theta_0,\theta_1\in L^{*,+}_{q_0}M$ and real numbers $s^-(q_0,\theta_1)<0<s^+(q_0,\theta_0)$ such that (cf. \cite{hintz2020inverse} or the arguments in~\S\ref{SsConf})
\begin{equation}
\label{Eqxminus}
x_1=\gamma_{q_0,\theta_1^\sharp}(s^-(q_0,\theta_1))\in (0,T)\times\partial N,\quad x_0=\gamma_{q_0,\theta_0^\sharp}(s^+(q_0,\theta_0))\in (0,T)\times\partial N,
\end{equation}
Denote
\[
\begin{split}
\xi_1&=\dot\gamma_{q_0,\theta_1^\sharp}(s^-(q_0,\theta_1))\in L^{+}_{x_1}M_1,\\
\xi_0&=-\dot\gamma_{q_0,\theta_0^\sharp}(s^+(q_0,\theta_0))\in L^{-}_{x_0}M_1.
\end{split}
\]
We can choose $\theta_0,\theta_1$ such that the two broken null geodesics $\gamma_{x_1,\xi_1}$ (starting from $x_1$ going forward in $\ft$) and $\gamma_{x_0,\xi_0}$ (starting from $x_0$ going backward in $\ft$) intersect only at $q_0$. Indeed, we claim that may arrange this by choosing $x_0,x_1$ as in the proof of Lemma~\ref{LemmaGeod} and using a small perturbation argument: the lemma already shows that they cannot intersect once more along $\gamma_{x_1,\xi_1}$ between $x_1$ and $q_0$ unless they are the same curves. This situation however can be avoided by perturbing $\theta_0$ slightly.

Choose local coordinates so that $g$ coincides with the Minkowski metric at $q_0$. Without loss of generality, one can then assume
\[
\theta_0=(-1,\pm\sqrt{1-r_0^2},r_0,0),\quad\quad \theta_1=(-1,1,0,0),
\]
for some $r_0\in [-1,1]$ but with $\theta_0\neq\theta_1$. Take a small parameter $\varsigma>0$ and introduce two perturbations of $\xi^{(1)}$
\[
\theta_2=(-1,\sqrt{1-\varsigma^2},\varsigma,0),\quad\quad \theta_3=(-1,\sqrt{1-\varsigma^2},-\varsigma,0).
\]
Thus, $\theta_2,\theta_3\in L^{*,+}_{q_0}M_1$. One can then write $\theta_0$ as a linear combination of $\theta_1,\theta_2,\theta_3$,
\[
\theta_0=\alpha_1\theta_1+\alpha_2\theta_2+\alpha_3\theta_3,
\]
with
\[
\alpha_1=\frac{-\sqrt{1-\varsigma^2}\pm\sqrt{1-r_0^2}}{1-\sqrt{1-\varsigma^2}},\quad \alpha_2=\frac{1\mp\sqrt{1-r_0^2}}{2(1-\sqrt{1-\varsigma^2})}+\frac{r_0}{2\varsigma},\quad \alpha_3=\frac{1\mp\sqrt{1-r_0^2}}{2(1-\sqrt{1-\varsigma^2})}-\frac{r_0}{2\varsigma}.
\]
We refer to \cite{hintz2020inverse} for more details.
Therefore, for $\kappa_0=1$ and $\kappa_j=-\alpha_j$, $j=1,2,3$, we have
\begin{equation}\label{kappa_linear}
\kappa_0\theta_0+\kappa_1\theta_1+\kappa_2\theta_2+\kappa_3\theta_3=0.
\end{equation}
By possibly perturbing $\theta_0$, we can assume that $\theta_0$ and $\theta_1$ are linearly independent. Then for $j=1,2,3$, $\kappa_j\neq 0$ and
\[
\lim_{\varsigma\rightarrow 0}|\kappa_j|=+\infty.
\]

For $j=2,3$, denote 
\[
x_j=\gamma_{q_0,\theta_j}(s^-(q_0,\theta_j)),
\]
\[
\xi_j=\dot\gamma_{q_0,\theta_j}(s^-(q_0,\theta_j))\in L^{+}_{x_j}M_1.
\]
Note that we can choose $\varsigma$ small enough so that $x_j\in (0,T)\times\partial N$ for $j=2,3$.

Now denote $\gamma^{(j)}=\gamma_{x_j,\xi_j}$ for $j=0,1,2,3$, to be the broken null-geodesics with $\gamma^{(j)}(0)=x_j$, $\dot{\gamma}^{(j)}(0)=\xi_j$ and $\gamma^{(j)}(s_j)=q_0$ for some $s_j>0$. For $j=1,2,3$, we define $\widetilde{u}_\rho^{(j)},u_\rho^{(j)}$ to be the Gaussian beam solutions which, before the first reflection, are of the form
\[
\begin{split}
\widetilde{u}_\rho^{(1)}=e^{\mathrm{i}\kappa_1\rho\varphi^{(1)}}\widetilde{\mathfrak{a}}^{(1)}_{\kappa_1\rho},\quad\widetilde{u}_\rho^{(2)}=e^{\mathrm{i}\kappa_2\rho\varphi^{(2)}}\widetilde{\mathfrak{a}}^{(2)}_{\kappa_2\rho},\quad\widetilde{u}_\rho^{(3)}=e^{\frac{\mathrm{i}}{2}\kappa_3\rho\varphi^{(3)}}\widetilde{\mathfrak{a}}^{(3)}_{\kappa_3\rho/2};\\
u_\rho^{(1)}=e^{\mathrm{i}\kappa_1\rho\varphi^{(1)}}\mathfrak{a}^{(1)}_{\kappa_1\rho},\quad u_\rho^{(2)}=e^{\mathrm{i}\kappa_2\rho\varphi^{(2)}}\mathfrak{a}^{(2)}_{\kappa_2\rho},\quad u_\rho^{(3)}=e^{\frac{\mathrm{i}}{2}\kappa_3\rho\varphi^{(3)}}\mathfrak{a}^{(3)}_{\kappa_3\rho/2}.
\end{split}
\]
Notice the $\frac{1}{2}$-factor which we put into the definition of $\widetilde{u}^{(3)}_\rho$ and $u^{(3)}_\rho$. For $j=1,2,3$, we can construct $\widetilde{u}^{(j)}=\widetilde{u}_\rho^{(j)}+\widetilde{R}_\rho^{(j)}$ and $u^{(j)}=u_\rho^{(j)}+R_\rho^{(j)}$ to be the solutions to \eqref{eq_tildeu} and \eqref{eq_u} with $f_{\rho,x,\xi}$ taken to be $f^{(j)}:=f_{\rho,x_j,\xi_j}=u_\rho^{(j)}\vert_{\partial M}$. We emphasize that $\widetilde{u}^{(j)}$ and $u^{(j)}$ satisfy the same boundary Dirichlet boundary conditions $\widetilde{u}^{(j)}\vert_{\partial M}=u^{(j)}\vert_{\partial M}=f^{(j)}$. 

Similarly, we can construct Gaussian beams of the form
\[
\widetilde{u}_\rho^{(0)}=e^{\mathrm{i}\kappa_0\rho\varphi^{(0)}}\widetilde{\mathfrak{a}}^{(0)}_{\kappa_0\rho},\quad u_\rho^{(0)}=e^{\mathrm{i}\kappa_0\rho\varphi^{(0)}}\mathfrak{a}^{(0)}_{\kappa_0\rho},
\]
before the first reflection, concentrating near $\gamma^{(0)}$. Then we can construct remainder terms $\widetilde{R}_\rho^{(0)}$, $R_\rho^{(0)}$ such that $\widetilde{u}^{(0)}=\widetilde{u}_\rho^{(0)}+\widetilde{R}_\rho^{(0)}$ and $u^{(0)}=u_\rho^{(0)}+R_\rho^{(0)}$ satisfy the backward wave equations
\begin{equation}\label{eq_tildeu0}
\begin{alignedat}{2}
  (\Box_g +q)\widetilde{u}^{(0)} &= 0, &\quad & \text{on }M,\\
  \widetilde{u}^{(0)} &= f^{(0)}(x), &\quad& \text{on }\partial M,\\
  \widetilde{u}^{(0)}(\ft,x') &= 0, &\quad& \ft>T,
\end{alignedat}
\end{equation}
and
\begin{equation}\label{eq_u0}
\begin{alignedat}{2}
  \Box_gu^{(0)} &= 0, &\quad & \text{on }M,\\
  u^{(0)} &= f^{(0)}(x), &\quad& \text{on }\partial M,\\
 u^{(0)}(\ft,x') &= 0, &\quad& \ft>T.
\end{alignedat}
\end{equation}
Here $f^{(0)}=u_\rho^{(0)}\vert_{\partial M}$, and $\delta$ is taken to be small enough such that $u_\rho^{(0)}=0$ near $\{\ft=T\}$.

We are now in a position to prove Proposition \ref{conf_uniqueness}. Notice that a fourth order linearization of the DN maps $\Lambda_{g,q,e^\beta\widetilde{a}}(f)=\Lambda_{g,0,a}(f)$, with $f=\epsilon_1f^{(1)}+\epsilon_2f^{(2)}+\epsilon_3f^{(3)}+\epsilon_4f^{(3)}$ (notice the repetition of $f^{(3)}$), leads to the integral identity
\[
\begin{split}
&\int_0^T\int_{\partial N}\frac{\partial^4}{\partial\epsilon_1\partial\epsilon_2\partial\epsilon_3\partial\epsilon_4}\Lambda_{g,q,e^\beta\widetilde{a}}(\epsilon_1f^{(1)}+\epsilon_2f^{(2)}+\epsilon_3f^{(3)}+\epsilon_4f^{(3)})f^{(0)}\mathrm{d}S\mathrm{d}t\\
=&\int_0^T\int_{\partial N}\frac{\partial^4}{\partial\epsilon_1\partial\epsilon_2\partial\epsilon_3\partial\epsilon_4}\Lambda_{g,0,a}(\epsilon_1f^{(1)}+\epsilon_2f^{(2)}+\epsilon_3f^{(3)}+\epsilon_4f^{(3)})f^{(0)}\mathrm{d}S\mathrm{d}t.
\end{split}
\]
and consequently
\begin{equation}\label{II_eq}
\widetilde{\mathcal{I}}:=\int_{M}e^{\beta}\widetilde{a}\widetilde{u}^{(1)}\widetilde{u}^{(2)}(\widetilde{u}^{(3)})^2\widetilde{u}^{(0)}\mathrm{d}V_g=\int_{M}au^{(1)}u^{(2)}(u^{(3)})^2u^{(0)}\mathrm{d}V_g=:\mathcal{I},
\end{equation}
via integration by parts.
Observe that
\begin{equation}\label{station1}
\begin{split}
&\rho^2\widetilde{\mathcal{I}}=\rho^2\int_M e^{\beta}\widetilde{a}e^{\mathrm{i}\rho S}\chi_1\chi_2\chi_3^2\chi_0(a_0^{(1)}+\rho^{-1}\kappa_1^{-1}\widetilde{a}_1^{(1)})(a_0^{(2)}+\rho^{-1}\kappa_2^{-1}\widetilde{a}_1^{(2)})(a_0^{(3)}+2\rho^{-1}\kappa_3^{-1}\widetilde{a}_1^{(3)})^2\\
&\hspace{8em}\times(a_0^{(0)}+\rho^{-1}\kappa_0^{-1}\widetilde{a}_1^{(0)})\mathrm{d}V_g+\mathcal{O}(\rho^{-2}).
\end{split}
\end{equation}
Here $\chi_j$ is the cutoff function associated with $\widetilde{u}_\rho^{(j)}(u_\rho^{(j)})$ as in \eqref{phi_a}.
The multiplication by $\rho^2$, resp.\ the $\mathcal{O}(\rho^{-2})$ remainder estimate are motivated by, resp.\ follow from the fact that $\int_M \vert e^{\mathrm{i}\rho S}\vert\mathrm{d}V_g=\mathcal{O}(\rho^{-2})$ and stationary phase, see~\eqref{EqStat} below. Here, we use that the phase function
\[
S:=\kappa_0\varphi^{(0)}+\kappa_1\varphi^{(1)}+\kappa_2\varphi^{(2)}+\kappa_3\varphi^{(3)},
\]
which satisfies the following properties:
\begin{lemma}[\text{\cite[Lemma 5]{feizmohammadi2019recovery}}]
 The function $S$ is well-defined in a neighborhood of $q_0$ and
 \begin{enumerate}
 \item $S(q_0)=0$;
 \item $\nabla S(q_0)=0$;
 \item $\Im S(x)\geq cd(x,q_0)^2$ for $x$ in a neighborhood of $q_0$, where $c>0$ is a constant.
 \end{enumerate}
 \end{lemma}
Similarly, we have
\begin{equation}\label{station2}
\begin{split}
&\rho^2\mathcal{I}=\rho^2\int_M ae^{\mathrm{i}\rho S}\chi_1\chi_2\chi_3^2\chi_0(a_0^{(1)}+\rho^{-1}\kappa_1^{-1}a_1^{(1)})(a_0^{(2)}+\rho^{-1}\kappa_2^{-1}a_1^{(2)})(a_0^{(3)}+2\rho^{-1}\kappa_3^{-1}a_1^{(3)})^2\\
&\hspace{8em}\times(a_0^{(0)}+\rho^{-1}\kappa_0^{-1}a_1^{(0)})\mathrm{d}V_g+\mathcal{O}(\rho^{-2}).
\end{split}
\end{equation}

By \eqref{II_eq}, \eqref{station1}, and \eqref{station2}, we obtain
\[
\rho^2\int_M e^{\beta}\widetilde{a}\chi_1\chi_2\chi_3^2\chi_0e^{\mathrm{i}\rho S}a_0^{(1)}a_0^{(2)}(a_0^{(3)})^2a_0^{(0)}\mathrm{d}V_g=\rho^2\int_M ae^{\mathrm{i}\rho S}\chi_1\chi_2\chi_3^2\chi_0a_0^{(1)}a_0^{(2)}(a_0^{(3)})^2a_0^{(0)}\mathrm{d}V_g+\mathcal{O}(\rho^{-1}).
\]
Since $\gamma^{(j)}$, $j=0,1,2,3$ only intersect at $q_0$, the function $\chi_1\chi_2\chi_3^2\chi_0$ is supported in a small neighborhood of $q_0$ and is equal to $1$ at $q_0$.
Using the method of stationary phase, we have
\begin{equation}
\label{EqStat}
\rho^2\int_M e^{\beta}\widetilde{a}e^{\mathrm{i}\rho S}\chi_1\chi_2\chi_3^2\chi_0a_0^{(1)}a_0^{(2)}(a_0^{(3)})^2a_0^{(0)}\mathrm{d}V_g=ce^{\beta(q_0)}\widetilde{a}(q_0)+\mathcal{O}(\rho^{-1}),
\end{equation}
and
\[
\rho^2\int_M ae^{\mathrm{i}\rho S}\chi_1\chi_2\chi_3^2\chi_0a_0^{(1)}a_0^{(2)}(a_0^{(3)})^2a_0^{(0)}\mathrm{d}V_g=ca(q_0)+\mathcal{O}(\rho^{-1})
\]
with some constant $c\neq 0$.
Thus, by letting $\rho\rightarrow+\infty$ we conclude that
\[
e^{\beta(q_0)}\widetilde{a}(q_0)=a(q_0).
\]
Since $q_0$ can be any point in $\mathbb{U}_g$, we have $\widetilde{a}=e^{-\beta}a$ in $\mathbb{U}_g$.

Next, we consider the recovery of $q$. We will adapt the arguments of \cite{feizmohammadi2019recovery}, where a similar problem is considered. We exclude the principal terms in \eqref{station1} and \eqref{station2} and get, upon multiplying by $\rho$,
\[
\begin{split}
\rho^2\int_M e^{\beta}\widetilde{a}e^{\mathrm{i}\rho S}\chi_1\chi_2\chi_3^2\chi_0\Big(&\kappa_1^{-1}\widetilde{a}_1^{(1)}a_0^{(2)}(a_0^{(3)})^2a_0^{(0)}+\kappa_2^{-1}\widetilde{a}_1^{(2)}a_0^{(1)}(a_0^{(3)})^2a_0^{(0)}\\
&\quad\quad+4\kappa_3^{-1}\widetilde{a}_1^{(3)}a_0^{(1)}a_0^{(2)}a_0^{(3)}a_0^{(0)}+\kappa_0^{-1}\widetilde{a}_1^{(0)}a_0^{(1)}a_0^{(2)}(a_0^{(3)})^2\Big)\\
=\rho^2\int_M ae^{\mathrm{i}\rho S}\chi_1\chi_2\chi_3^2\chi_0\Big(&\kappa_1^{-1}a_1^{(1)}a_0^{(2)}(a_0^{(3)})^2a_0^{(0)}+\kappa_2^{-1}a_1^{(2)}a_0^{(1)}(a_0^{(3)})^2a_0^{(0)}\\
&\quad\quad+4\kappa_3^{-1}a_1^{(3)}a_0^{(1)}a_0^{(2)}a_0^{(3)}a_0^{(0)}+\kappa_0^{-1}a_1^{(0)}a_0^{(1)}a_0^{(2)}(a_0^{(3)})^2\Big)+\mathcal{O}(\rho^{-1}).
\end{split}
\]
Letting $\rho\rightarrow+\infty$ and applying the method of stationary phase again, we obtain
\[
\begin{split}
&e^{\beta(q_0)}\widetilde{a}(q_0)\left(\frac{c_{1,1}}{\kappa_1}\widetilde{a}_1^{(1)}(q_0)+\frac{c_{1,2}}{\kappa_2}\widetilde{a}_1^{(2)}(q_0)+\frac{c_{1,3}}{\kappa_3}\widetilde{a}_1^{(3)}(q_0)+\frac{c_{1,0}}{\kappa_0}\widetilde{a}_1^{(0)}(q_0)\right)\\
&\quad=a(q_0)\left(\frac{c_{1,1}}{\kappa_1}a_1^{(1)}(q_0)+\frac{c_{1,2}}{\kappa_2}a_1^{(2)}(q_0)+\frac{c_{1,3}}{\kappa_3}a_1^{(3)}(q_0)+\frac{c_{1,0}}{\kappa_0}a_1^{(0)}(q_0)\right).
\end{split}
\]
where $c_{1,j}$, $j=0,1,2,3$, are non-zero constants that do not depend on $q$ (thus do not depend on $\beta$). For more details, we refer to \cite{feizmohammadi2019recovery}. Recall now that $e^{\beta(q_0)}\widetilde{a}(q_0)=a(q_0)$ and
\[
\widetilde{a}_{1}^{(j)}(q_0)-a_{1}^{(j)}(q_0)=-\frac{\mathrm{i}}{2}\det(Y(s_j))^{-\frac{1}{2}}\int_{0}^{s_j}q(\gamma^{(j)}(s))\mathrm{d}s
\]
for $j=0,1,2,3$. Arguing as in \cite{feizmohammadi2019recovery}, one can then obtain
\[
\sum_{j=0}^3\frac{c_{1,j}}{\kappa_j}\int_0^{s_j}q(\gamma^{(j)}(s))\mathrm{d}s=0.
\]
Letting $\varsigma\rightarrow 0$, we have
\[
\int_{0}^{s_0}q(\gamma^{(0)}(s))\mathrm{d}s=0,
\]
since $\kappa_0=1$ and $\lim_{\varsigma\rightarrow 0}\kappa_j^{-1}=0$ for $j=1,2,3$ recalling the definition of $\kappa_j$ around \eqref{kappa_linear}. Finally, by differentiating in $s_0$ we obtain $q(q_0)=0$. Since $-e^{-\beta}q=\Box_g e^{-\beta}$, we finally conclude that
\[
\Box_g e^{-\beta}=0\quad\text{in }\mathbb{U}_g.
\]
This finishes the proof of Proposition \ref{conf_uniqueness}.\\

For the proof of Theorem \ref{ThmI}, there is one more remark we would like to make. Notice that in Proposition \ref{conf_uniqueness}, we assumed $\widetilde{g}=e^{-2\beta}g$ in $M$, while in Section \ref{interiordetermination} we have only proved that $\widetilde{g}=e^{-2\beta}g$ in $\mathbb{U}_g=\mathbb{U}_{\widetilde{g}}$ (up to diffeomorphism). However, one can easily adapt the proof by using the property of finite speed of propagation. We shall leave Proposition \ref{conf_uniqueness} and Lemma \ref{lemma_q} as they are, since they might be of independent interest.

\section*{Acknowledgements}
 Part of this research was conducted during the period PH served as a Clay Research Fellow; PH also gratefully acknowledges support from the NSF under Grant No.\ DMS-1955614 and a Sloan Research Fellowship. GU was partially supported by NSF, a Walker Professorship at UW and a Si-Yuan Professorship at IAS, HKUST. JZ was partially supported by Research Grant Council of Hong Kong (GRF grant 16305018). PH, GU as a Clay Senior Scholar, and JZ acknowledge the warm hospitality of MSRI in the fall of 2019, where part of this work was carried out.

\bibliographystyle{abbrv}
\bibliography{biblio}

\begin{thebibliography}{10}

\bibitem{alinhac1995non}
S.~Alinhac and M.~S. Baouendi.
\newblock A non uniqueness result for operators of principal type.
\newblock {\em Mathematische Zeitschrift}, 220(1):561--568, 1995.

\bibitem{assylbekov2020inverse}
Y.~Assylbekov and T.~Zhou.
\newblock {I}nverse problems for nonlinear {M}axwell's equations with second
  harmonic generation.
\newblock {\em Preprint, arXiv:2009.03467}, 2020.

\bibitem{assylbekov2017direct}
Y.~M. Assylbekov and T.~Zhou.
\newblock Direct and inverse problems for the nonlinear time-harmonic {M}axwell
  equations in {K}err-type media.
\newblock {\em to appear in J. Spectral Theory, arXiv:1709.07767}, 2017.

\bibitem{balehowsky2020inverse}
T.~Balehowsky, A.~Kujanp{\"a}{\"a}, M.~Lassas, and T.~Liimatainen.
\newblock An inverse problem for the relativistic {B}oltzmann equation.
\newblock {\em Preprint, arXiv:2011.09312}, 2020.

\bibitem{bao2014sensitivity}
G.~Bao and H.~Zhang.
\newblock Sensitivity analysis of an inverse problem for the wave equation with
  caustics.
\newblock {\em Journal of the American Mathematical Society}, 27(4):953--981,
  2014.

\bibitem{belishev1996boundary}
M.~Belishev and A.~Katchalov.
\newblock Boundary control and quasiphotons in the problem of reconstruction of
  a {R}iemannian manifold via dynamic data.
\newblock {\em Journal of Mathematical Sciences}, 79(4):1172--1190, 1996.

\bibitem{caday2019scattering}
P.~Caday, M.~V. de~Hoop, V.~Katsnelson, and G.~Uhlmann.
\newblock Scattering control for the wave equation with unknown wave speed.
\newblock {\em Archive for Rational Mechanics and Analysis}, 231(1):409--464,
  2019.

\bibitem{carstea2020inverse}
C.~I. C{\^a}rstea.
\newblock On an inverse boundary value problem for a nonlinear time-harmonic
  {M}axwell system.
\newblock {\em Journal of Inverse and Ill-posed Problems}, 1(ahead-of-print),
  2020.

\bibitem{carstea2021inverse}
C.~I. C{\^a}rstea, A.~Feizmohammadi, Y.~Kian, K.~Krupchyk, and G.~Uhlmann.
\newblock The {C}alder{\'o}n inverse problem for isotropic quasilinear
  conductivities.
\newblock {\em arXiv preprint arXiv:2103.05917}, 2021.

\bibitem{carstea2019reconstruction}
C.~I. C{\^a}rstea, G.~Nakamura, and M.~Vashisth.
\newblock Reconstruction for the coefficients of a quasilinear elliptic partial
  differential equation.
\newblock {\em Applied Mathematics Letters}, 98:121--127, 2019.

\bibitem{chen2019detection}
X.~Chen, M.~Lassas, L.~Oksanen, and G.~P. Paternain.
\newblock Detection of {H}ermitian connections in wave equations with cubic
  non-linearity.
\newblock {\em Preprint, arXiv:1902.05711}, 2019.

\bibitem{chen2020inverse}
X.~Chen, M.~Lassas, L.~Oksanen, and G.~P. Paternain.
\newblock Inverse problem for the {Y}ang--{M}ills equations.
\newblock {\em Preprint, arXiv:2005.12578}, 2020.

\bibitem{de2018nonlinear}
M.~de~Hoop, G.~Uhlmann, and Y.~Wang.
\newblock Nonlinear interaction of waves in elastodynamics and an inverse
  problem.
\newblock {\em Mathematische Annalen}, 376(1-2):765--795, 2020.

\bibitem{de2019nonlinear}
M.~V. de~Hoop, G.~Uhlmann, and Y.~Wang.
\newblock Nonlinear responses from the interaction of two progressing waves at
  an interface.
\newblock {\em Annales de l'Institut Henri Poincar{\'e} C, Analyse non
  lin{\'e}aire}, 36(2):347--363, 2019.

\bibitem{dos2016calderon}
D.~Dos Santos~Ferreira, Y.~Kurylev, M.~Lassas, and M.~Salo.
\newblock The {C}alder{\'o}n problem in transversally anisotropic geometries.
\newblock {\em Journal of the European Mathematical Society},
  18(11):2579--2626, 2016.

\bibitem{duistermaat1996fourier}
J.~J. Duistermaat.
\newblock {\em Fourier integral operators}, volume~2.
\newblock Springer, 1996.

\bibitem{DuistermaatHormanderFIO2}
J.~J. Duistermaat and L.~H\"ormander.
\newblock Fourier integral operators. {II}.
\newblock {\em Acta Mathematica}, 128(1):183--269, 1972.

\bibitem{feizmohammadi2019timedependent}
A.~Feizmohammadi, J.~Ilmavirta, Y.~Kian, and L.~Oksanen.
\newblock Recovery of time dependent coefficients from boundary data for
  hyperbolic equations.
\newblock {\em to appear in J. Spectr. Theory, arXiv:1901.04211}, 2019.

\bibitem{feizmohammadi2020inverse}
A.~Feizmohammadi, M.~Lassas, and L.~Oksanen.
\newblock Inverse problems for non-linear hyperbolic equations with disjoint
  sources and receivers.
\newblock {\em Preprint, arXiv:2006.12158}, 2020.

\bibitem{feizmohammadi2019recovery}
A.~Feizmohammadi and L.~Oksanen.
\newblock Recovery of zeroth order coefficients in non-linear wave equations.
\newblock {\em Preprint, arXiv:1903.12636}, 2019.

\bibitem{feizmohammadi2019inverse}
A.~Feizmohammadi and L.~Oksanen.
\newblock An inverse problem for a semi-linear elliptic equation in
  {R}iemannian geometries.
\newblock {\em Journal of Differential Equations}, 269(6):4683--4719, 2020.

\bibitem{hintz2017reconstruction}
P.~Hintz and G.~Uhlmann.
\newblock Reconstruction of {L}orentzian manifolds from boundary light
  observation sets.
\newblock {\em International Mathematics Research Notices},
  2019(22):6949--6987, 2019.

\bibitem{hintz2020inverse}
P.~Hintz, G.~Uhlmann, and J.~Zhai.
\newblock An inverse boundary value problem for a semilinear wave equation on
  {L}orentzian manifolds.
\newblock {\em Preprint, arXiv:2005.10447}, 2020.

\bibitem{katchalov1998multidimensional}
A.~Katchalov and Y.~Kurylev.
\newblock Multidimensional inverse problem with incomplete boundary spectral
  data.
\newblock {\em Communications in Partial Differential Equations},
  23(1-2):27--59, 1998.

\bibitem{kenig2014calderon}
C.~Kenig and M.~Salo.
\newblock The {C}alder{\'o}n problem with partial data on manifolds and
  applications.
\newblock {\em Analysis \& PDE}, 6(8):2003--2048, 2014.

\bibitem{kian2020partial}
Y.~Kian, K.~Krupchyk, and G.~Uhlmann.
\newblock Partial data inverse problems for quasilinear conductivity equations.
\newblock {\em arXiv preprint arXiv:2010.11409}, 2020.

\bibitem{krupchyk2019partial}
K.~Krupchyk and G.~Uhlmann.
\newblock Partial data inverse problems for semilinear elliptic equations with
  gradient nonlinearities.
\newblock {\em to appear in Mathematical Research Letters, arXiv:1909.08122},
  2019.

\bibitem{krupchyk2020inverse}
K.~Krupchyk and G.~Uhlmann.
\newblock Inverse problems for nonlinear magnetic schr{\"o}dinger equations on
  conformally transversally anisotropic manifolds.
\newblock {\em arXiv preprint arXiv:2009.05089}, 2020.

\bibitem{krupchyk2020remark}
K.~Krupchyk and G.~Uhlmann.
\newblock A remark on partial data inverse problems for semilinear elliptic
  equations.
\newblock {\em Proceedings of the American Mathematical Society},
  148(2):681--685, 2020.

\bibitem{kurylev2014inverse}
Y.~Kurylev, M.~Lassas, L.~Oksanen, and G.~Uhlmann.
\newblock Inverse problem for {E}instein-scalar field equations.
\newblock {\em arXiv:1406.4776}, 2014.

\bibitem{kurylev2018inverse}
Y.~Kurylev, M.~Lassas, and G.~Uhlmann.
\newblock Inverse problems for {L}orentzian manifolds and non-linear hyperbolic
  equations.
\newblock {\em Inventiones Mathematicae}, 212(3):781--857, 2018.

\bibitem{lai2021reconstruction}
R.-Y. Lai, G.~Uhlmann, and Y.~Yang.
\newblock Reconstruction of the collision kernel in the nonlinear boltzmann
  equation.
\newblock {\em SIAM Journal on Mathematical Analysis}, 53(1):1049--1069, 2021.

\bibitem{lassas2019partial}
M.~Lassas, T.~Liimatainen, Y.-H. Lin, and M.~Salo.
\newblock Partial data inverse problems and simultaneous recovery of boundary
  and coefficients for semilinear elliptic equations.
\newblock {\em Revista Matem{\'a}tica Iberoamericana}, 2020.

\bibitem{lassas2019inverse}
M.~Lassas, T.~Liimatainen, Y.-H. Lin, and M.~Salo.
\newblock Inverse problems for elliptic equations with power type
  nonlinearities.
\newblock {\em Journal de math{\'e}matiques pures et appliqu{\'e}es},
  145:44--82, 2021.

\bibitem{lassas2020uniqueness}
M.~Lassas, T.~Liimatainen, L.~Potenciano-Machado, and T.~Tyni.
\newblock Uniqueness and stability of an inverse problem for a semi-linear wave
  equation.
\newblock {\em Preprint, arXiv:2006.13193}, 2020.

\bibitem{lassas2017determination}
M.~Lassas, G.~Uhlmann, and Y.~Wang.
\newblock Determination of vacuum space-times from the {E}instein-{M}axwell
  equations.
\newblock {\em Preprint, arXiv:1703.10704}, 2017.

\bibitem{lassas2018inverse}
M.~Lassas, G.~Uhlmann, and Y.~Wang.
\newblock Inverse problems for semilinear wave equations on {L}orentzian
  manifolds.
\newblock {\em Communications in Mathematical Physics}, 360(2):555--609, 2018.

\bibitem{melrose1979lagrangian}
R.~B. Melrose and G.~A. Uhlmann.
\newblock Lagrangian intersection and the {C}auchy problem.
\newblock {\em Communications on Pure and Applied Mathematics}, 32(4):483--519,
  1979.

\bibitem{ONeillSemi}
B.~O'Neill.
\newblock {\em {S}emi-{R}iemannian geometry with applications to relativity},
  volume 103.
\newblock Academic press, 1983.

\bibitem{petrov2016einstein}
A.~Z. Petrov.
\newblock {\em Einstein spaces}.
\newblock Elsevier, 2016.

\bibitem{RalstonLocalized}
J.~V. Ralston.
\newblock Solutions of the wave equation with localized energy.
\newblock {\em Communications on Pure and Applied Mathematics}, 22(6):807--823,
  1969.

\bibitem{stefanov2018inverse}
P.~Stefanov and Y.~Yang.
\newblock The inverse problem for the {D}irichlet-to-{N}eumann map on
  {L}orentzian manifolds.
\newblock {\em Analysis \& PDE}, 11(6):1381--1414, 2018.

\bibitem{tzou2021point}
L.~Tzou.
\newblock Determining {R}iemannian manifolds from nonlinear wave observations
  at a single point.
\newblock {\em Preprint, arXiv:2102.01841}, 2021.

\bibitem{uhlmann2018determination}
G.~Uhlmann and Y.~Wang.
\newblock Determination of space-time structures from gravitational
  perturbations.
\newblock {\em Communications on Pure and Applied Mathematics},
  73(6):1315--1367, 2020.

\bibitem{uhlmann2019inverse}
G.~Uhlmann and J.~Zhai.
\newblock On an inverse boundary value problem for a nonlinear elastic wave
  equation.
\newblock {\em Preprint, arXiv:1912.11756}, 2019.

\bibitem{uhlmann2020inverse}
G.~Uhlmann and J.~Zhai.
\newblock Inverse problems for nonlinear hyperbolic equations.
\newblock {\em Discrete \& Continuous Dynamical Systems-A}, 41(1):455, 2020.

\bibitem{wang2019inverse}
Y.~Wang and T.~Zhou.
\newblock Inverse problems for quadratic derivative nonlinear wave equations.
\newblock {\em Communications in Partial Differential Equations},
  44(11):1140--1158, 2019.

\end{thebibliography}

\end{document}